\documentclass{article}
\pdfoutput=1
\usepackage{geometry}
\usepackage{amsmath,amssymb,amsthm,graphicx,float,caption,subcaption,bm}
\usepackage{romannum}
\usepackage{bbm}
\usepackage{graphicx} 
\usepackage{mathrsfs}
\usepackage{stmaryrd}
\usepackage{verbatim}
\usepackage{lipsum}
\newcommand\blfootnote[1]{%
	\begingroup
	\renewcommand\thefootnote{}\footnote{#1}%
	\addtocounter{footnote}{-1}%
	\endgroup
}

\allowdisplaybreaks[4]
\newtheorem{thm}{Theorem}[section]
\newtheorem{lem}[thm]{Lemma}
\newtheorem{assumption}[thm]{Assumption}
\newtheorem{prop}[thm]{Proposition}
\newtheorem{defn}[thm]{Definition}
\newtheorem{rem}[thm]{Remark}

\title{The obstacle problem for stochastic porous media equations}
\author{Ruoyang~Liu\footnotemark[1] \and Shanjian~Tang\footnotemark[2]}

\begin{document}
\maketitle
\pagenumbering{arabic}
\begin{abstract}
We prove the existence and uniqueness of non-negative entropy solutions
of the obstacle problem for stochastic porous media equations. The
core of the method is to combine the entropy formulation with the
penalization method.

\end{abstract}

\renewcommand{\thefootnote}{\fnsymbol{footnote}}
\blfootnote{\textit{MSC2020 subject classifications}: 60H15;35R60;35K59.}
\footnotetext[1]{School of Mathematical Sciences, Fudan University, Shanghai 200433, China.
    Email:18110180044@fudan.edu.cn}
\footnotetext[2]{Institute of Mathematical Finance and Department of Finance and
Control Sciences, School of Mathematical Sciences, Fudan University,
Shanghai 200433, P. R. China. Partially supported by National Science
Foundation of China (Grant Nos. 11631004 and 12031009). Email:sjtang@fudan.edu.cn.}

\section{Introduction}

Consider the following obstacle problem with an upper obstacle $S$:
\begin{equation}
\begin{cases}
du=\big[\Delta\Phi(u)+F(t,x,u)\big]dt-\nu(dt,x)\\
\quad\quad\quad+{\displaystyle \sum_{k=1}^{\infty}}\sigma^{k}(u)dW_{t}^{k},\quad(t,x)\in(0,T)\times\mathbb{T}^{d};\\
u(t,x)\leq S(t),\quad d\mathbb{P}\otimes dt\otimes dx\mathrm{-a.e.};\\
u(0,x)=\xi(x),\quad x\in\mathbb{T}^{d};\\
\int_{Q_{T}}(u-S)\nu(dtdx)=0,\quad\mathrm{a.s.}\ \omega\in\Omega,
\end{cases}\label{eq:main equation_start}
\end{equation}
where $\mathbb{T}^{d}$ is $d$-dimensional torus, and $\{W_{\cdot}^{k}\}_{k\in\mathbb{N}^{+}}$
is a sequence of independent Brownian motions. $\Phi$ is a monotone
function, and a typical type is $\Phi(u)=|u|^{m-1}u$ with $m>1$.
The solution of (\ref{eq:main equation_start}) is a pair $(u,\nu)$.

The initial physical model of this work is fluid flow in a container
with a limitation on the density of the fluid. That is, the least
amount of the fluid will be pumped out of the container, which makes
sure that the density of the fluid is lower than the limitation $S$.

Porous media equations arise in the flow of an ideal gas through a
homogeneous porous medium, and the solution $u$ is the scaled density
of the gas \cite{muskat1938flow}. These equations have applications
in various fields, such as population dynamics \cite{gurtin1977diffusion}
and the theory of ionized gases at high temperature \cite{zel2002physics}.
Since there are quite a lot of studies on these equations, we only
introduce relevant works, and other results can be found in \cite{prevot2007concise,vazquez2007porous,barbu2016stochastic}
and references therein. 

By transforming into a porous media equation with random coefficients,
\cite{lototsky2007random,barbu2011random} proved the existence and
uniqueness result for the equation with linear multiplicative noise.
In \cite{ren2007stochastic,barbu2008existence,barbu2009existence,barbu2015stochastic},
with the monotone operator method \cite{pardoux1975equation,krylov1979stochastic,prevot2007concise,vazquez2007porous},
they obtained the well-posedness under the condition that the diffusion
$\sigma^{k}$ is Lipschitz continuous in $H^{-1}$. Under the condition
$m>2$ and the Lipschitz continuity of $\sigma^{k}$, \cite{bauzet2015degenerate}
used a entropy formulation to prove the well-posedness, and \cite{gess2018well}
gave the existence and uniqueness of the kinetic solution of the stochastic
porous media equations. When $m>1$ and $\sigma^{k}$ has linear growth
and locally $1/2$-H\"older continuity, \cite{dareiotis2019entropy}
obtained the existence and uniqueness of the entropy solution on torus
with a probabilistic approach. Using a weighted $L^{1}$-norm, \cite{dareiotis2020ergodicity}
extended the results to the bounded domain. 

Obstacle problems for deterministic partial differential equations
have been studied extensively in the early stage using variational
inequality (see \cite{mignot1977inequations} and references therein).
\cite{korte2009obstacle,baroni2014lorentz,bogelein2015obstacle,bogelein2017holder,cho2020holder,schatzler2020obstacle}
generalized to the porous media equations.  Avelin \cite{avelin2017comparison}
proposed the potential theory for porous media equation, and proved
that the smallest supersolution is also a variational weak solution.
\cite{korte2019lower} proved the existence of supersolutionsunder
weaken conditions on the obstacle.

Haussmann and Pardoux \cite{haussmann1989stochastic} firstly studied
the obstacle problem for stochastic heat equation on the interval
$[0,1]$ by stochastic variational inequalities. Nualart and Pardoux
\cite{nualart1992white} gave the existence of solution of the heat
equation driven by the space-time white noise using the penalization
method, while \cite{donati1993white,XU2009White} proved for general
diffusion term. However, these works only considered the special obstacle
$S\equiv0$. Yang and Tang \cite{yang2013dynkin} used the penalization
method on the backward equation with two obstacles. In order to deal with general obstacle, \cite{denis2014obstacle,yang2019obstacle,dong2019obstacle}
studied the quasilinear equations using the parabolic potential theory
\cite{pierre1979problems,pierre1980representant}. Qiu \cite{qiu2014quasi}
expanded to backward stochastic partial differential equations. \cite{matoussi2010obstacle,matoussi2015obstacle,matoussi2017backward,klimsiak2018obstacle,dong2020obstacle,denis2020quasilinear}
applied the method of probabilistic interpretation of the solution
using backward doubly stochastic differential equation. It is worth
noting that the probabilistic interpretation method is still feasible
for nonlinear stochastic partial differential equations. 

Our objective is to study the well-posedness of non-negative solution
of the obstacle problem for stochastic porous media equations. A major
technical difficulty encountered is that we cannot directly apply
It\^{o}'s formula on the entropy solution $u_{\epsilon}$ of the penalized
equation (see (\ref{eq:penalized equation})), which is necessary
to a priori estimates of both $u_{\epsilon}$ and the penalty term.
To overcome this difficulty, we merge the penalization method with
the $L_{1}$ technique of stochastic porous media equations. We follow
\cite{dareiotis2019entropy} to approximate $\Phi$ with $\Phi_{n}$
in the penalized equation, which is nondegenerate and thus has a unique
$L_{2}$-solution $u_{n,\epsilon}$ (see Theorem \ref{thm:exuni_approximation equation}).
Furthermore, the $L_{2}$ norm of the penalty term can be estimated
if the the difference $u_{n,\epsilon}-S$ has bounded variation when
$u_{n,\epsilon}=S$. This estimate ensures the existence of the weak
limit $\nu$ in the entropy solution (see Definition \ref{def:entropy solution with ob}).
Moreover, using $L_{1}$ technique, the existence of $u$ comes from the
limit of $u_{n,\epsilon}$. The uniqueness of the entropy solution
$(u,\nu)$ is derived from a direct $L_{1}$ estimate as in \cite{dareiotis2019entropy}.
To our best knowledge, this is the first study to the obstacle problem
under the entropy formulation of degenerate stochastic partial differential
equations. 

This paper is organized as follows. Section 2 states the main theorem
after introducing notations and assumptions to formulate the entropy
solution. We also prove the non-negativity of the entropy solution.
In Section 3, we approximate the equation by non-degenerate ones,
and obtain the well-posedness of $L_{2}$-solution $u_{n,\epsilon}$
to the penalized equations. A priori estimates for both $u_{n,\epsilon}$
and penalty term are derived. In Section 4, we introduce some Lemma
and prove the $(\star)$-property of $u_{n,\epsilon}$. Then the $L_{1}^{+}$
estimates are given for two different entropy solutions in Section
5. In Section 6, we pass to the limit $n\rightarrow\infty$ and then
$\epsilon\rightarrow0^{+}$ to acquire the existence of the entropy
solution $(u,\nu)$. To prove the uniqueness, we give another $L_{1}$
estimate which can reduce the limitation on $(\star)$-property.

\section{Entropy formulation}

We firstly introduce some notations and settings. Let $(\Omega,\mathcal{F},\{\mathcal{F}_{t}\},\mathbb{P})$
be a complete filtered probability space and $\mathcal{P}$ be the
predictable $\sigma$-algebra generated by $\{\mathcal{F}_{t}\}$.
The noise $W=\{W_{t}^{k}:t\in[0,\infty),\ k\in\mathbb{N}^{+}\}$ is
a sequence of independent ${\mathcal{F}_{t}}$-adapted Wiener processes
on $\Omega$. For fixed $T>0$, denote $Q_{T}:=[0,T]\times\mathbb{T}^{d}$.
$L_{p}$ and $H_{p}^{s}$ are the usual Lebesgue and Sobolev space
with $p\geq2$ and $s>0$. When $p=2$, we simplify $H_{2}^{s}$ as
$H^{s}$ (cf. \cite{evans1998partial}). Given the obstacle $S$,
the obstacle problem denoted by $\Pi_{S}(\Phi,F,\xi)$ is to seek
a pair $(u,\nu)$ such that
\begin{equation}
\begin{cases}
du=\big[\Delta\Phi(u)+F(t,x,u)\big]dt-\nu(dt,x)\\
\quad\quad\quad+{\displaystyle \sum_{k=1}^{\infty}}\sigma^{k}(u)dW_{t}^{k},\quad(t,x)\in(0,T)\times\mathbb{T}^{d};\\
u(t,x)\leq S(t),\quad d\mathbb{P}\otimes dt\otimes dx-\textrm{a.e.};\\
u(0,x)=\xi(x),\quad x\in\mathbb{T}^{d};\\
\int_{Q_{T}}(u-S)\nu(dtdx)=0,\quad\textrm{a.s.}\ \omega\in\Omega.
\end{cases}\label{eq:main equation}
\end{equation}
The nonlinear function $\Phi$ is of porous media type. The measure
$\nu$ is introduced to ensure that $u(t,x)\leq S(t)$, and the last
condition is the so-called Skohorod condition which requires that
the force $\nu$ we apply to the equation is ``minimal''.

We denote by $\Pi(\Phi,F,\xi)$ the following stochastic porous media
equation:
\[
\begin{cases}
du=\big[\Delta\Phi(u)+F(t,x,u)\big]dt+{\displaystyle \sum_{k=1}^{\infty}}\sigma^{k}(u)dW_{t}^{k},\quad(t,x)\in(0,T)\times\mathbb{T}^{d};\\
u(0,x)=\xi(x),\quad x\in\mathbb{T}^{d}.
\end{cases}
\]
The well-posedness of the entropy solution of $\Pi(\Phi,0,\xi)$ is
available in \cite{dareiotis2019entropy}.

Define $Q_{T}:=[0,T]\times\mathbb{T}^{d}$. Given a smooth function
$\rho:\mathbb{R}\rightarrow[0,2]$, which is supported in $(0,1)$
and integrates to $1$. For $\theta>0$, we set $\rho_{\theta}(r):=\theta^{-1}\rho(\theta^{-1}r)$
as a sequence of mollifiers. For any function $g:\mathbb{R}\rightarrow\mathbb{R}$,
we use the notation 
\[
\left\llbracket g\right\rrbracket (r):=\int_{0}^{r}g(s)ds,\quad r\in\mathbb{R}.
\]
Define the set of functions
\[
\mathcal{E}:=\left\{ \eta\in C^{2}(\mathbb{R}):\eta\ \mathrm{is\ convex\ with}\ \eta^{\prime\prime}\ \mathrm{compactly\ supported}\right\} .
\]
Fix two constants $\kappa\in(0,1/2]$ and $\bar{\kappa}\in(0,1)$.
For fixed $m>1$, there exists constants $K\geq1$ and $N_{0}\geq0$ such that the
following assumptions hold:
\begin{assumption}
\label{assu:assumption for phi}The function $\Phi:\mathbb{R\rightarrow R}$
is differentiable, strictly increasing and odd. The function $\zeta(r):=\sqrt{\Phi^{\prime}(r)}$
is differentiable away from the origin such that
\[
\left|\zeta(0)\right|\leq K,\ \ \ \left|\zeta^{\prime}(r)\right|\leq K\left|r\right|^{\frac{m-3}{2}},\quad\forall r\in(0,\infty)
\]
and
\[
K\zeta(r)\geq\mathbf{1}_{\{\left|r\right|\geq1\}},\ \ K\left|\left\llbracket \zeta\right\rrbracket (r)-\left\llbracket \zeta\right\rrbracket (s)\right|\geq\begin{cases}
\left|r-s\right|, & \textrm{if}\ \left|r\right|\lor\left|s\right|\geq1;\\
\left|r-s\right|^{\frac{m+1}{2}}, & \textrm{if}\ \left|r\right|\lor\left|s\right|<1.
\end{cases}
\]
\end{assumption}

\begin{assumption}
\label{assu:assumption for xi}The initial condition $\xi\geq0$ is
an $\mathcal{F}_{0}$-measurable $L_{m+1}(\mathbb{T}^{d})$-valued
random variable such that $\mathbb{E}\left\Vert \xi\right\Vert _{L_{m+1}(\mathbb{T}^{d})}^{m+1}<\infty$.
\end{assumption}

\begin{assumption}
\label{assu:assumption for sigma}The function $\sigma:\mathbb{R}\mapsto l_{2}$
satisfies $\sigma(0)=\mathbf{0}$ and
\[
|\sigma(r)-\sigma(s)|_{l_{2}}\leq K(|r-s|^{1/2+\kappa}+|r-s|),\quad\forall r,s\in\mathbb{R}.
\]

\end{assumption}

\begin{assumption}
\label{assu:inhomogeneous function}The function $F:Q_{T}\times\mathbb{R}\rightarrow\mathbb{R}$
satisfies $F(t,x,0)=0$ for any $(t,x)\in Q_{T}$, and
\[
|F(t,x_{1},r_{1})-F(t,x_{2},r_{2})|\leq K|x_{1}-x_{2}|^{\kappa}+N_{0}|r_{1}-r_{2}|
\]
\end{assumption}

\begin{assumption}
\label{assu:Assuption for barrier}The obstacle $S$ satisfies the
following equation
\begin{equation}
\begin{cases}
dS=h_{S}dt+{\displaystyle \sum_{k=1}^{\infty}}\sigma^{k}(S)dW_{t}^{k},\quad t\in[0,T];\\
S(0)=S_{0},
\end{cases}\label{eq:assforS1}
\end{equation}
where $h_{S}\in L_{2}(\Omega_{T})$ and $S_{0}\in L_{2}(\Omega)$.
Further, $S(t)\geq0,\ \forall t\in[0,T]$ and 
\[
S_{0}\geq\xi(x),\quad\forall(\omega,x)\in\Omega\times\mathbb{T}^{d}.
\]
\end{assumption}

\begin{rem}
It is natural that the functions $\sigma^{k}(\cdot)$ and $F(t,x,\cdot)$ vanish at zero
in Assumptions \ref{assu:assumption for sigma} and \ref{assu:inhomogeneous function}
 since the equation $\Pi(\Phi,F,\xi)$ describes the density of the gas
flow through a porous media. In particular, $u\equiv0$ is a solution of $\Pi(\Phi,F,0)$.
Moreover, Assumption \ref{assu:assumption for sigma} also yields the linear growth:
\[
|\sigma(r)|_{l_{2}}\leq K(1+|r|),\quad\forall r\in\mathbb{R}.
\]
\end{rem}

\begin{rem}
In Assumption \ref{assu:Assuption for barrier}, if $h_{S}\geq0$
and $S_{0}\geq0$, the barrier $S$ which satisfies (\ref{eq:assforS1})
is non-negative. Moreover, a constant barrier $S$ satisfies Assumption
\ref{assu:Assuption for barrier} if $\sigma(S)=\mathbf{0}$.
\end{rem}

\begin{rem}
Assumption \ref{assu:Assuption for barrier} is strong enough such
that the measure $\nu$ is absolutely continuous with respect to Lebesgue
measure, and for convenience, we still denote by
$\nu$ the density function.
\end{rem}

Define $\Omega_{T}:=\Omega\times[0,T]$. We now introduce the definition
of the entropy solution.
\begin{defn}
\label{def:entropy solution without}An entropy solution of the stochastic
porous media equation $\Pi(\Phi,F,\xi)$ is a predictable stochastic
process $u:\Omega_{T}\rightarrow L_{m+1}(\mathbb{T}^{d})$ such that

(i)\ $u\in L_{m+1}(\Omega_{T};L_{m+1}(\mathbb{T}^{d}))$;

(ii)\ For all $f\in C_{b}(\mathbb{R})$, we have $\left\llbracket \zeta f\right\rrbracket (u)\in L_{2}(\Omega_{T};H^{1}(\mathbb{T}^{d}))$
and
\[
\partial_{x_{i}}\left\llbracket \zeta f\right\rrbracket (u)=f(u)\partial_{x_{i}}\left\llbracket \zeta\right\rrbracket (u);
\]

(iii)\ For all $(\eta,\varphi,\varrho)\in\mathcal{E}\times C_{c}^{\infty}([0,T))\times C^{\infty}(\mathbb{T}^{d})$
and $\phi:=\varphi\varrho\geq0$, we have almost surely
\begin{align}
 & -\int_{0}^{T}\int_{\mathbb{T}^{d}}\eta(u)\partial_{t}\phi dxdt\nonumber \\
 & \leq\int_{\mathbb{T}^{d}}\eta(\xi)\phi(0)dx+\int_{0}^{T}\int_{\mathbb{T}^{d}}\llbracket \zeta^2  \eta^{\prime} \rrbracket(u)\Delta\phi dxdt+\int_{0}^{T}\int_{\mathbb{T}^{d}}\eta^{\prime}(u)F(t,x,u)\phi dxdt\label{eq:entropy formula withou} \\
 & \quad+\int_{0}^{T}\int_{\mathbb{T}^{d}}\left(\frac{1}{2}\eta^{\prime\prime}(u)\sum_{k=1}^{\infty}|\sigma^{k}(u)|^{2}\phi-\eta^{\prime\prime}(u)|\nabla\left\llbracket \zeta\right\rrbracket (u)|^{2}\phi\right)dxdt\nonumber \\
 & \quad+\sum_{k=1}^{\infty}\int_{0}^{T}\int_{\mathbb{T}^{d}}\eta^{\prime}(u)\phi\sigma^{k}(u)dxdW_{t}^{k}.\nonumber 
\end{align}
\end{defn}

\begin{defn}
\label{def:entropy solution with ob}An entropy solution of the obstacle
problem $\Pi_{S}(\Phi,F,\xi)$ is a pair $(u,\nu)$ such that

(i)\ The functions $u$ and $\nu$ are two predictable stochastic
processes and satisfy $(u,\nu)\in L_{m+1}(\Omega_{T};L_{m+1}(\mathbb{T}^{d}))\times L_{2}(\Omega_{T};L_{2}(\mathbb{T}^{d}))$
and $\nu\geq0$;

(ii)\ For all $f\in C_{b}(\mathbb{R})$, we have $\left\llbracket \zeta f\right\rrbracket (u)\in L_{2}(\Omega_{T};H^{1}(\mathbb{T}^{d}))$
and
\[
\partial_{x_{i}}\left\llbracket \zeta f\right\rrbracket (u)=f(u)\partial_{x_{i}}\left\llbracket \zeta\right\rrbracket (u);
\]

(iii)\ For all $(\eta,\varphi,\varrho)\in\mathcal{E}\times C_{c}^{\infty}([0,T))\times C^{\infty}(\mathbb{T}^{d})$
and $\phi:=\varphi\varrho\geq0$, we have almost surely
\begin{align}
 & -\int_{0}^{T}\int_{\mathbb{T}^{d}}\eta(u)\partial_{t}\phi dxdt\nonumber \\
 & \leq\int_{\mathbb{T}^{d}}\eta(\xi)\phi(0)dx+\int_{0}^{T}\int_{\mathbb{T}^{d}}\llbracket \zeta^2  \eta^{\prime} \rrbracket(u)\Delta\phi dxdt\nonumber \\
 & \quad+\int_{0}^{T}\int_{\mathbb{T}^{d}}\eta^{\prime}(u)(F(t,x,u)-\nu)\phi dxdt\label{eq:entropy formula-1}\\
 & \quad+\int_{0}^{T}\int_{\mathbb{T}^{d}}\left(\frac{1}{2}\eta^{\prime\prime}(u)\sum_{k=1}^{\infty}|\sigma^{k}(u)|^{2}\phi-\eta^{\prime\prime}(u)|\nabla\left\llbracket \zeta\right\rrbracket (u)|^{2}\phi\right)dxdt\nonumber \\
 & \quad+\sum_{k=1}^{\infty}\int_{0}^{T}\int_{\mathbb{T}^{d}}\eta^{\prime}(u)\phi\sigma^{k}(u)dxdW_{t}^{k};\nonumber 
\end{align}

(iv)\ We have $u\leq S$ almost everywhere in $Q_{T}$, almost surely,
and the following Skohorod condition holds 
\[
\int_{Q_{T}}(u-S)\nu dtdx=0,\quad\textrm{a.s.\ }\omega\in\Omega.
\]
\end{defn}

Our main result is stated as follows.
\begin{thm}
\label{thm:maintheorem}Let Assumptions \ref{assu:assumption for phi}-\ref{assu:Assuption for barrier}
hold. Then, there exists a unique entropy solution $(u,\nu)$ to $\Pi_{S}(\Phi,F,\xi)$.
Moreover, if $(\tilde{u},\tilde{\nu})$ is the entropy solution of
$\Pi_{S}(\Phi,F,\tilde{\xi})$, we have
\[
\underset{t\in[0,T]}{\mathrm{ess\ sup}}\ \mathbb{E}\left\Vert u(t)-\tilde{u}(t)\right\Vert _{L_{1}(\mathbb{T}^{d})}\leq C\mathbb{E}\left\Vert \xi-\tilde{\xi}\right\Vert _{L_{1}(\mathbb{T}^{d})}
\]
for a constant $C$ depending only on $K$, $N_{0}$, $d$ and $T$.
\end{thm}

\begin{rem}
The same assertion holds true for the lower barrier case under the
additional conditions 
\[
S_{0}\in L_{m+1}(\Omega)\cap L_{4}(\Omega),\quad h_{S}\in L_{m+1}(\Omega_{T})\cap L_{4}(\Omega;L_{2}(0,T)).
\]
In fact, applying It\^{o}'s formula to calculate the terms
\[
\left\Vert u-S\right\Vert _{L_{2}(\mathbb{T}^{d})}^{2},\quad\left\Vert u-S\right\Vert _{L_{m+1}(\mathbb{T}^{d})}^{m+1},\quad\left\Vert u-S\right\Vert _{L_{2}(\mathbb{T}^{d})}^{2}\cdot|S|^{m-1},
\]
\[
\left\Vert u-S-1\right\Vert _{L_{2}(\mathbb{T}^{d})}^{2},\quad\int_{\mathbb{T}^{d}}\int_{0}^{u}\Phi(r)drdx,\quad\text{and\ }\left\Vert (u-S)^{-}\right\Vert _{L_{2}(\mathbb{T}^{d})}^{2},
\]
 we obtain a priori estimates in Section \ref{sec:Approximation}
with $p\leq4$, which are sufficient conditions for Theorem \ref{thm:maintheorem}.
\end{rem}

\begin{prop}
\label{prop:nonnegative}Under Assumptions \ref{assu:assumption for phi}-\ref{assu:Assuption for barrier},
if $(u,\nu)$ is the entropy solution of $\Pi_{S}(\Phi,F,\xi)$, we
have $u\geq0$ almost everywhere in $Q_{T}$, almost surely.
\end{prop}

\begin{proof}
For sufficiently small $\delta>0$, we introduce a function $\eta_{\delta}\in C^{2}(\mathbb{R})$
defined by
\[
\eta_{\delta}(0)=\eta_{\delta}^{\prime}(0)=0,\quad\eta_{\delta}^{\prime\prime}(r)=\rho_{\delta}(r).
\]
Applying entropy formulation (\ref{eq:entropy formula withou}) with
$\eta(\cdot)=\eta_{\delta}(-\cdot)$ and $\phi$ independent of $x$,
we get
\begin{align}
 & -\mathbb{E}\int_{0}^{T}\int_{\mathbb{T}^{d}}\eta_{\delta}(-u)\partial_{t}\phi dxdt\nonumber \\
 & \leq\mathbb{E}\int_{0}^{T}\int_{\mathbb{T}^{d}}-\eta_{\delta}^{\prime}(-u)F(t,x,u)\phi dxdt+\mathbb{E}\int_{0}^{T}\int_{\mathbb{T}^{d}}\eta_{\delta}^{\prime}(-u)\nu\phi dxdt\label{eq:entropyfornonnegative}\\
 & \quad+\mathbb{E}\int_{0}^{T}\int_{\mathbb{T}^{d}}\frac{1}{2}\eta_{\delta}^{\prime\prime}(-u)\sum_{k=1}^{\infty}|\sigma^{k}(u)|^{2}\phi-\eta_{\delta}^{\prime\prime}(-u)|\nabla\left\llbracket \zeta\right\rrbracket (u)|^{2}\phi dxdt.\nonumber 
\end{align}
In view of the Skohorod condition and the non-negativity of $\nu$
and $S$, we have
\begin{align*}
 & \mathbb{E}\int_{0}^{T}\int_{\mathbb{T}^{d}}\eta_{\delta}^{\prime}(-u)\nu\phi dxdt\\
 & =\mathbb{E}\int_{0}^{T}\int_{\mathbb{T}^{d}}\mathbf{1}_{\{\nu=0\}}\eta_{\delta}^{\prime}(-u)\nu\phi dxdt+\mathbb{E}\int_{0}^{T}\int_{\mathbb{T}^{d}}\mathbf{1}_{\{\nu>0\}}\eta_{\delta}^{\prime}(-S)\nu\phi dxdt=0.
\end{align*}
Combining inequality (\ref{eq:entropyfornonnegative}) with Assumptions \ref{assu:assumption for sigma} and \ref{assu:inhomogeneous function} and $|\eta_{\delta}^{\prime}(r)\cdot r-r^{+}|\leq\delta$,
we have
\[
-\mathbb{E}\int_{0}^{T}\int_{\mathbb{T}^{d}}\eta_{\delta}(-u)\partial_{t}\phi dxdt\leq N_{0}\mathbb{E}\int_{0}^{T}\int_{\mathbb{T}^{d}}(-u)^{+}\phi dxdt+C\delta^{2\kappa}.
\]
Since
\[
\left|\mathbb{E}\int_{0}^{T}\int_{\mathbb{T}^{d}}\eta_{\delta}(-u)\partial_{t}\phi dxdt-\mathbb{E}\int_{0}^{T}\int_{\mathbb{T}^{d}}(-u)^{+}\partial_{t}\phi dxdt\right|\leq C\delta,
\]
we get
\[
-\mathbb{E}\int_{0}^{T}\int_{\mathbb{T}^{d}}(-u)^{+}\partial_{t}\phi dxdt\leq C\mathbb{E}\int_{0}^{T}\int_{\mathbb{T}^{d}}(-u)^{+}\phi dxdt+C\delta^{2\kappa}
\]
for sufficiently small $\delta>0$ and a constant $C$ which is independent
of $\delta$. Setting $\delta\rightarrow0^{+}$, as the proof of (\ref{eq:ineq_givenphis}), using Gr\"onwall's inequality, we have
\[
\mathbb{E}\int_{\mathbb{T}^{d}}(-u(t,x))^{+}dx\leq0,\quad\textrm{a.e.}\ t\in[0,T].
\]
Therefore, we have $u\geq0$ almost everywhere in
$Q_{T}$, almost surely.
\end{proof}

\section{Approximation\label{sec:Approximation}}

A natural method to deal with the obstacle problem is to consider
the penalized equation
\begin{equation}
\begin{cases}
du_{\epsilon}=\big[\Delta\Phi(u_{\epsilon})+F(t,x,u_{\epsilon})-{\displaystyle \frac{1}{\epsilon}}(u_{\epsilon}-S)^{+}\big]dt\\
\quad\quad\quad+{\displaystyle \sum_{k=1}^{\infty}}\sigma^{k}(u_{\epsilon})dW_{t}^{k},\quad(t,x)\in(0,T)\times\mathbb{T}^{d};\\
u_{\epsilon}(0,x)=\xi(x),\quad x\in\mathbb{T}^{d}.
\end{cases}\label{eq:penalized equation}
\end{equation}
We expect that both $u_{\epsilon}$ and $(u_{\epsilon}-S)^{+}/\epsilon$
have limits $u$ and $\nu$ when $\epsilon\rightarrow0^{+}$, and
the pair $(u,\nu)$ is a solution of $\Pi_{S}(\Phi,F,\xi)$. However,
for the entropy solutions of the stochastic porous media equations,
the lack of uniform estimates to the penalty term $(u_{\epsilon}-S)^{+}/\epsilon$
will make it difficult to get the existence of the limit $\nu$. To
solve this problem, we use a sequence of smooth functions $\{\Phi_{n}\}_{n\in\mathbb{N}}$
to approximate $\Phi$ as in \cite{dareiotis2019entropy} With the
well-posedness and properties of solutions of penalized equations,
we prove the existence and estimates of $u_{\epsilon}$ and $(u_{\epsilon}-S)^{+}/\epsilon$.
\begin{prop}
\label{prop:Assumptio for coef}\cite[Proposition 5.1]{dareiotis2019entropy}
Let $\Phi$ satisfy Assumption \ref{assu:assumption for phi} with
a constant $K>1$. Then, for all $n\in\mathbb{N}$, there exists an
increasing function $\Phi_{n}\in C^{\infty}(\mathbb{R})$ with bounded
derivatives, satisfying Assumption \ref{assu:assumption for phi}
with constant $3K$, such that $\zeta_{n}(r)\geq2/n$, and
\[
\sup_{|r|\leq n}|\zeta(r)-\zeta_{n}(r)|\leq4/n.
\]
\end{prop}

Define 
\begin{equation}\label{defn for xi_n}
\xi_{n}:=\rho_{1/n}^{\otimes d}\ast\big((-n)\vee(\xi\land n)\big).
\end{equation}
Then $\xi_{n}$ also satisfies Assumption \ref{assu:assumption for xi}
and \ref{assu:Assuption for barrier}. For any $\epsilon>0$, we define
the penalty term
\[
G_{\epsilon}(r,\tilde{r}):=\frac{(r-\tilde{r})^{+}}{\epsilon},
\]
which is Lipschitz continuous with Lipschitz constant $1/\epsilon$
in both $r$ and $\tilde{r}$. Moreover, the non-negativity of the
barrier $S$ indicates that $G_{\epsilon}(0,S(\omega,t))\equiv0$ almost surely in $\Omega_T$.
In this section, we study the penalized equation $\Pi(\Phi_{n},F-G_{\epsilon}(\cdot,S),\xi_{n})$
which reads,

\[
\begin{cases}
du_{n,\epsilon}=\big[\Delta\Phi_{n}(u_{n,\epsilon})+F(t,x,u_{n,\epsilon})-G_{\epsilon}(u_{n,\epsilon},S)\big]dt\\
\quad\quad\quad+{\displaystyle \sum_{k=1}^{\infty}}\sigma^{k}(u_{n,\epsilon})dW_{t}^{k},\quad(t,x)\in Q_{T};\\
u_{n,\epsilon}(0,x)=\xi_{n}(x),\quad x\in\mathbb{T}^{d}.
\end{cases}
\]
 
\begin{defn}
\label{def:L_2 solution}An $L_{2}$-solution of equation $\Pi(\Phi_{n},F-G_{\epsilon}(\cdot,S),\xi_{n})$
is a continuous $L_{2}(\mathbb{T}^{d})$-valued process $u_{n,\epsilon}$,
such that $u_{n,\epsilon}\in L_{2}(\Omega_{T};H^{1}(\mathbb{T}^{d}))$,
$\nabla\Phi_{n}(u_{n,\epsilon})\in L_{2}(\Omega_{T};L_{2}(\mathbb{T}^{d}))$,
and for all $\phi\in C^{\infty}(\mathbb{T}^{d})$, we have
\begin{align*}
\int_{\mathbb{T}^d}u_{n,\epsilon}(t,x)\phi dx & =\int_{\mathbb{T}^d}\xi_{n}\phi dx-\int_{0}^{t}\Bigg[\int_{\mathbb{T}^d}\nabla\Phi_{n}(u_{n,\epsilon})\nabla\phi dx\\
 & \quad+\int_{\mathbb{T}^d}\big[ F(s,x,u_{n,\epsilon})-G_{\epsilon}(u_{n,\epsilon},S)\big]\phi dx\Bigg]ds\\
 & \quad-\sum_{k=1}^{\infty}\int_{0}^{t}\int_{\mathbb{T}^d}\sigma^{k}(u_{n,\epsilon})\phi dx dW_{s}^{k},\quad\textrm{a.e.}\ t\in[0,T].
\end{align*}

\end{defn}

We first prove a priori estimates of $u_{n,\epsilon}$.
\begin{thm}
\label{thm:esimate for u_n,epsilon}Let Assumptions \ref{assu:assumption for phi}-\ref{assu:Assuption for barrier}
hold. Then, for all $n\in\mathbb{N}$, $\epsilon>0$ and $p\in[2,\infty)$,
there exists a constant $C$ independent of $n$ and $\epsilon$ such
that
\begin{align}
 & \mathbb{E}\sup_{t\leq T}\left\Vert u_{n,\epsilon}(t)\right\Vert _{L_{2}(\mathbb{T}^{d})}^{p}+\mathbb{E}\left\Vert \nabla\left\llbracket \zeta_{n}\right\rrbracket (u_{n,\epsilon})\right\Vert _{L_{2}(Q_{T})}^{p}\label{eq:priori estimate1} \\
 & \quad+(\frac{1}{\epsilon})^{\frac{p}{2}}\mathbb{E}\left\Vert (u_{n,\epsilon}-S)^{+}\right\Vert _{L_{2}(Q_{T})}^{p}\leq C\left(1+\mathbb{E}\left\Vert \xi_{n}\right\Vert _{L_{2}(\mathbb{T}^{d})}^{p}\right),\ \text{and}\nonumber
\end{align}
\begin{align}
 & \mathbb{E}\sup_{t\leq T}\left\Vert u_{n,\epsilon}(t)\right\Vert _{L_{m+1}(\mathbb{T}^{d})}^{m+1}+\frac{1}{\epsilon}\mathbb{E}\int_{0}^{T}\int_{\mathbb{T}^{d}}|(u_{n,\epsilon}-S)^{+}|^{2}|u_{n,\epsilon}|^{m-1}dxds\label{eq:priori estimate2} \\
 & \leq C\left(1+\mathbb{E}\left\Vert \xi_{n}\right\Vert _{L_{m+1}(\mathbb{T}^{d})}^{m+1}\right).\nonumber
\end{align}
\end{thm}

\begin{proof}
Applying It\^{o}'s formula (cf. \cite[Lemma 2]{dareiotis2015boundedness}),
we have
\begin{align*}
\left\Vert u_{n,\epsilon}(t)\right\Vert _{L_{2}(\mathbb{T}^{d})}^{2} & =\left\Vert \xi_{n}\right\Vert _{L_{2}(\mathbb{T}^{d})}^{2}-2\int_{0}^{t}\langle\partial_{x_{i}}\Phi_{n}(u_{n,\epsilon}),\partial_{x_{i}}u_{n,\epsilon}\rangle_{L_{2}(\mathbb{T}^{d})}ds\\
 & \quad+2\int_{0}^{t}\langle F(s,x,u_{n,\epsilon})-\frac{1}{\epsilon}(u_{n,\epsilon}-S)^{+},u_{n,\epsilon}\rangle_{L_{2}(\mathbb{T}^{d})}ds\\
 & \quad+2\sum_{k=1}^{\infty}\int_{0}^{t}\langle\sigma^{k}(u_{n,\epsilon}),u_{n,\epsilon}\rangle_{L_{2}(\mathbb{T}^{d})}dW_{s}^{k}\\
 & \quad+\int_{0}^{t}\sum_{k=1}^{\infty}\left\Vert \sigma^{k}(u_{n,\epsilon})\right\Vert _{L_{2}(\mathbb{T}^{d})}^{2}ds,\quad\textrm{a.e.}\ t\in[0,T].
\end{align*}
In view of the definition of $\zeta_{n}$ and Assumptions \ref{assu:assumption for sigma}
and \ref{assu:inhomogeneous function}, we have
\begin{align*}
\left\Vert u_{n,\epsilon}(t)\right\Vert _{L_{2}(\mathbb{T}^{d})}^{2} & \leq C+\left\Vert \xi_{n}\right\Vert _{L_{2}(\mathbb{T}^{d})}^{2}+C\int_{0}^{t}\left\Vert u_{n,\epsilon}\right\Vert _{L_{2}(\mathbb{T}^{d})}^{2}ds\\
 & \quad-2\int_{0}^{t}\left\Vert \nabla\left\llbracket \zeta_{n}\right\rrbracket (u_{n,\epsilon})\right\Vert _{L_{2}(\mathbb{T}^{d})}^{2}+\frac{1}{\epsilon}\langle(u_{n,\epsilon}-S)^{+},u_{n,\epsilon}\rangle_{L_{2}(\mathbb{T}^{d})}ds\\
 & \quad+2\sum_{k=1}^{\infty}\int_{0}^{t}\langle\sigma^{k}(u_{n,\epsilon}),u_{n,\epsilon}\rangle_{L_{2}(\mathbb{T}^{d})}dW_{s}^{k}.
\end{align*}
Since the barrier $S$ is non-negative, we have
\[
-\frac{2}{\epsilon}\int_{0}^{t}\int_{\mathbb{T}^d}u_{n,\epsilon}(u_{n,\epsilon}-S)^{+}dxds\leq-\frac{2}{\epsilon}\int_{0}^{t}\left\Vert (u_{n,\epsilon}-S)^{+}\right\Vert _{L_{2}(\mathbb{T}^{d})}^{2}ds.
\]
Raising to the power $p/2$, taking suprema up to time $t'$ and expectations,
gives
\begin{align}
 & \mathbb{E}\sup_{t\leq t'}\left\Vert u_{n,\epsilon}(t)\right\Vert _{L_{2}(\mathbb{T}^{d})}^{p}+\mathbb{E}\left(\int_{0}^{t'}\left\Vert \nabla\left\llbracket \zeta_{n}\right\rrbracket (u_{n,\epsilon})\right\Vert _{L_{2}(\mathbb{T}^{d})}^{2}ds\right)^{\frac{p}{2}}\nonumber \\
 & \quad+\mathbb{E}\left(\frac{1}{\epsilon}\int_{0}^{t'}\left\Vert (u_{n,\epsilon}-S)^{+}\right\Vert _{L_{2}(\mathbb{T}^{d})}^{2}ds\right)^{\frac{p}{2}}\nonumber \\
 & \leq C\Bigg[1+\mathbb{E}\left\Vert \xi_{n}\right\Vert _{L_{2}(\mathbb{T}^{d})}^{p}+\int_{0}^{t'}\mathbb{E}\sup_{t\leq s}\left\Vert u_{n,\epsilon}(t)\right\Vert _{L_{2}(\mathbb{T}^{d})}^{p}ds\nonumber\\
 & \quad+\mathbb{E}\sup_{t\leq t'}\left|\sum_{k=1}^{\infty}\int_{0}^{t}\langle\sigma^{k}(u_{n,\epsilon}),u_{n,\epsilon}\rangle_{L_{2}(\mathbb{T}^{d})}dW_{s}^{k}\right|^{\frac{p}{2}}\Bigg].\nonumber 
\end{align}
Since
\begin{align*}
 & \mathbb{E}\sup_{t\leq t'}\left|\sum_{k=1}^{\infty}\int_{0}^{t}\langle\sigma^{k}(u_{n,\epsilon}),u_{n,\epsilon}\rangle_{L_{2}(\mathbb{T}^{d})}dW_{s}^{k}\right|^{\frac{p}{2}}\\
 & \leq C\mathbb{E}\left|\int_{0}^{t'}\left(\int_{\mathbb{T}^{d}}\sum_{k=1}^{\infty}|\sigma^{k}(u_{n,\epsilon})|^{2}dx\right)\left(\int_{\mathbb{T}^{d}}|u_{n,\epsilon}|^{2}dx\right)ds\right|^{\frac{p}{4}}\\
 & \leq C\mathbb{E}\left|1+\int_{0}^{t'}\left(\int_{\mathbb{T}^{d}}|u_{n,\epsilon}|^{2}dx\right)^{2}ds\right|^{\frac{p}{4}}\\
 & \leq C+\bar{\varepsilon}C\mathbb{E}\sup_{t\leq t'}\left\Vert u_{n,\epsilon}(t)\right\Vert _{L_{2}(\mathbb{T}^{d})}^{p}+\frac{C}{\bar{\varepsilon}}\int_{0}^{t'}\mathbb{E}\sup_{t\leq s}\left\Vert u_{n,\epsilon}(t)\right\Vert _{L_{2}(\mathbb{T}^{d})}^{p}ds
\end{align*}
for sufficiently small $\bar{\varepsilon}>0$, applying Gr\"onwall's
inequality, we have the first desired estimate (\ref{eq:priori estimate1}).

We now prove inequality (\ref{eq:priori estimate2}). Using It\^{o}'s
formula (cf. \cite[Lemma 2]{dareiotis2015boundedness}), we have 
\begin{align*}
 & \left\Vert u_{n,\epsilon}(t)\right\Vert _{L_{m+1}(\mathbb{T}^{d})}^{m+1}\\
 & =\left\Vert \xi_{n}\right\Vert _{L_{m+1}(\mathbb{T}^{d})}^{m+1}-(m^{2}+m)\int_{0}^{t}\int_{\mathbb{T}^{d}}\partial_{x_{i}}\Phi_{n}(u_{n,\epsilon})\left|u_{n,\epsilon}\right|^{m-1}\partial_{x_{i}}u_{n,\epsilon}dxds\\
 & \quad+(m+1)\int_{0}^{t}\int_{\mathbb{T}^{d}}\left[F(s,x,u_{n,\epsilon})-\frac{1}{\epsilon}(u_{n,\epsilon}-S)^{+}\right]\cdot u_{n,\epsilon}\left|u_{n,\epsilon}\right|^{m-1}dxds\\
 & \quad+(m^{2}+m)\sum_{k=1}^{\infty}\int_{0}^{t}\int_{\mathbb{T}^{d}}\sigma^{k}(u_{n,\epsilon})u_{n,\epsilon}\cdot\left|u_{n,\epsilon}\right|^{m-1}dxdW_{s}^{k}\\
 & \quad+\frac{(m^{2}+m)}{2}\int_{0}^{t}\sum_{k=1}^{\infty}\int_{\mathbb{T}^{d}}\big|\sigma^{k}(u_{n,\epsilon})\big|^{2}\left|u_{n,\epsilon}\right|^{m-1}dxds.
\end{align*}
Since the obtacle $S$ is non-negative and $\Phi_{n}$ is monotone,
in view of Assumptions \ref{assu:assumption for sigma} and \ref{assu:inhomogeneous function},
we have
\begin{align*}
 & \left\Vert u_{n,\epsilon}(t)\right\Vert _{L_{m+1}(\mathbb{T}^{d})}^{m+1}+\frac{m+1}{\epsilon}\int_{0}^{t}\int_{\mathbb{T}^{d}}|(u_{n,\epsilon}-S)^{+}|^{2}\left|u_{n,\epsilon}\right|^{m-1}dxds \\
 & \leq C+\left\Vert \xi_{n}\right\Vert _{L_{m+1}(\mathbb{T}^{d})}^{m+1}+C\int_{0}^{t}\int_{\mathbb{T}^{d}}\left|u_{n,\epsilon}\right|^{m+1}dxds\\
 & \quad+(m^{2}+m)\sum_{k=1}^{\infty}\int_{0}^{t}\int_{\mathbb{T}^{d}}\sigma^{k}(u_{n,\epsilon})u_{n,\epsilon}\cdot\left|u_{n,\epsilon}\right|^{m-1}dxdW_{s}^{k}. 
\end{align*}
Since
\begin{align*}
 & (m^{2}+m)\mathbb{E}\sup_{t\leq t'}\left|\sum_{k=1}^{\infty}\int_{0}^{t}\int_{\mathbb{T}^{d}}\sigma^{k}(u_{n,\epsilon})u_{n,\epsilon}\cdot\left|u_{n,\epsilon}\right|^{m-1}dxdW_{s}^{k}\right|\\
 & \leq C\mathbb{E}\left[\left(\int_{0}^{t'}\sum_{k=1}^{\infty}\left(\int_{\mathbb{T}^{d}}\sigma^{k}(u_{n,\epsilon})u_{n,\epsilon}\cdot\left|u_{n,\epsilon}\right|^{m-1}dx\right)^{2}ds\right)^{\frac{1}{2}}\right]\\
 & \leq C+C\mathbb{E}\left[\left(\int_{0}^{t'}\left\Vert u_{n,\epsilon}(s)\right\Vert _{L_{m+1}(\mathbb{T}^{d})}^{2(m+1)}ds\right)^{\frac{1}{2}}\right]\\
 & \leq C+\frac{1}{2}\mathbb{E}\sup_{t\leq t'}\left\Vert u_{n,\epsilon}(t)\right\Vert _{L_{m+1}(\mathbb{T}^{d})}^{m+1}+C\mathbb{E}\left[\int_{0}^{t'}\left\Vert u_{n,\epsilon}(s)\right\Vert _{L_{m+1}(\mathbb{T}^{d})}^{m+1}ds\right],
\end{align*}
using Gr\"onwall's inequality, we obtain the second desired inequality
(\ref{eq:priori estimate2}).
\end{proof}
\begin{rem}
\label{rem:L2equalentropy}In view of Definition \ref{def:L_2 solution},
Theorem \ref{thm:esimate for u_n,epsilon} and the smoothness of $\zeta_{n}$,
with It\^{o}'s formula, the $L_{2}$-solution $u_{n,\epsilon}$ is also
an entropy solution of $\Pi(\Phi_{n},F-G_{\epsilon}(\cdot,S),\xi_{n})$
in the sense of Definition \ref{def:entropy solution without}.
\end{rem}

\begin{lem}
\label{Lem:nonnegative of u_n,e}Let Assumptions \ref{assu:assumption for phi}-\ref{assu:Assuption for barrier}
hold. Then, for all $n\in\mathbb{N}$ and $\epsilon>0$, if $u_{n,\epsilon}$
is an $L_{2}$-solution of $\Pi(\Phi_{n},F-G_{\epsilon}(\cdot,S),\xi_{n})$,
we have $u_{n,\epsilon}\geq0$ almost everywhere in $Q_{T}$, almost
surely.
\end{lem}

\begin{proof}
For sufficiently small $\delta>0$, we introduce a function $\eta_{\delta}\in C^{2}(\mathbb{R})$
defined by
\[
\eta_{\delta}(0)=\eta_{\delta}^{\prime}(0)=0,\quad\eta_{\delta}^{\prime\prime}(r)=\rho_{\delta}(r).
\]
Based on Remark \ref{rem:L2equalentropy}, applying entropy formulation
(\ref{eq:entropy formula withou}) with $\eta(\cdot)=\eta_{\delta}(-\cdot)$
and $\phi$ which is independent of $x$, we have
\begin{align*}
 & -\mathbb{E}\int_{0}^{T}\int_{\mathbb{T}^{d}}\eta_{\delta}(-u_{n,\epsilon})\partial_{t}\phi dxdt\\
 & \leq\mathbb{E}\int_{0}^{T}\int_{\mathbb{T}^{d}}\eta_{\delta}^{\prime}(-u_{n,\epsilon})\bigg[\frac{1}{\epsilon}(u_{n,\epsilon}-S)^{+}-F(t,x,u_{n,\epsilon})\bigg]\phi dxdt\\
 & \quad+\mathbb{E}\int_{0}^{T}\int_{\mathbb{T}^{d}}\frac{1}{2}\eta_{\delta}^{\prime\prime}(-u_{n,\epsilon})\sum_{k=1}^{\infty}|\sigma^{k}(u_{n,\epsilon})|^{2}\phi-\eta_{\delta}^{\prime\prime}(-u_{n,\epsilon})|\nabla\left\llbracket \zeta_{n}\right\rrbracket (u_{n,\epsilon})|^{2}\phi dxdt.
\end{align*}
Since $\mathrm{supp}\ \eta_{\delta}^{\prime}(-\cdot)\subset(\infty,0]$
and the barrier $S$ is non-negative, we have
\[
\eta_{\delta}^{\prime}(-u_{n,\epsilon}(t,x))\cdot(u_{n,\epsilon}(t,x)-S(t))^{+}=0,\quad\textrm{a.s.}\ (\omega,t,x)\in\Omega\times Q_{T}.
\]
Therefore, proceeding as in the proof of Proposition \ref{prop:nonnegative},
we have
\[
\mathbb{E}\int_{\mathbb{T}^{d}}(-u_{n,\epsilon}(t,x))^{+}dx\leq0\quad\textrm{a.e.}\ t\in[0,T].
\]
Then, the proof is complete.
\end{proof}
\begin{thm}
Let Assumptions \ref{assu:assumption for phi}-\ref{assu:Assuption for barrier}
hold. Then, for all $n\in\mathbb{N}$ and $\epsilon>0$, we have
\begin{equation}
\mathbb{E}\left\Vert \nabla\Phi_{n}(u_{n,\epsilon})\right\Vert _{L_{2}(Q_{T})}^{2}\leq C\left(1+\mathbb{E}\left\Vert \xi_{n}\right\Vert _{L_{m+1}(\mathbb{T}^{d})}^{m+1}\right)\label{prioriestimateforphi}
\end{equation}
for a constant $C$ independent of $n$ and $\epsilon$.
\end{thm}

\begin{proof}
Applying It\^{o}'s formula (cf. \cite{krylov2013relatively}), we have
\begin{align}
 & \int_{\mathbb{T}^{d}}\int_{0}^{u_{n,\epsilon}(t)}\Phi_{n}(r)drdx\nonumber\\
 & =\int_{\mathbb{T}^{d}}\int_{0}^{\xi_{n}}\Phi_{n}(r)drdx-\int_{0}^{t}\int_{\mathbb{T}^{d}}\partial_{x_{i}}\Phi_{n}(u_{n,\epsilon})\partial_{x_{i}}\Phi_{n}(u_{n,\epsilon})dxds\nonumber \\
 & \quad+\int_{0}^{t}\int_{\mathbb{T}^{d}}\left(F(s,x,u_{n,\epsilon})-\frac{1}{\epsilon}(u_{n,\epsilon}-S)^{+}\right)\Phi_{n}(u_{n,\epsilon})dxds\label{eq:second estimateforphi} \\
 & \quad+\frac{1}{2}\int_{0}^{t}\sum_{k=1}^{\infty}\int_{\mathbb{T}^{d}}|\sigma^{k}(u_{n,\epsilon})|^{2}\Phi_{n}^{\prime}(u_{n,\epsilon})dxds\nonumber \\
 & \quad+\sum_{k=1}^{\infty}\int_{0}^{t}\int_{\mathbb{T}^{d}}\sigma^{k}(u_{n,\epsilon})\Phi_{n}(u_{n,\epsilon})dxdW_{s}^{k}.\nonumber 
\end{align}
From Assumption \ref{assu:assumption for phi}, we have
\[
\Phi_{n}^{\prime}(r)=\zeta^{2}(r)\leq\left(\left|\zeta(0)\right|+\int_{0}^{r}\left|\zeta^{\prime}(t)\right|dt\right)^{2}\leq C(1+\left|r\right|^{m-1}).
\]
Then, from Assumptions \ref{assu:assumption for sigma} and \ref{assu:inhomogeneous function},
we have
\begin{align}
 & \int_{\mathbb{T}^{d}}\int_{0}^{\xi_{n}}\Phi_{n}(r)drdx+\int_{0}^{T}\int_{\mathbb{T}^{d}}F(s,x,u_{n,\epsilon})\Phi_{n}(u_{n,\epsilon})dxds
 \label{eq:second estimate for init F and 1/2}\\
 & +\int_{0}^{t}\sum_{k=1}^{\infty}\int_{\mathbb{T}^{d}}\left|\sigma^{k}(u_{n,\epsilon})\right|^{2}\Phi_{n}^{\prime}(u_{n,\epsilon})dxds\leq  C\left(1+\left\Vert \xi_{n}\right\Vert _{L_{m+1}(\mathbb{T}^{d})}^{m+1}+\left\Vert u_{n,\epsilon}\right\Vert _{L_{m+1}(Q_{T})}^{m+1}\right)\nonumber
\end{align}
For the last term in the right hand side of (\ref{eq:second estimateforphi}),
applying Burkholder-Davis-Gundy inequality and H\"older's inequality,
we have
\begin{align}
 & \mathbb{E}\sup_{t\leq T}\left|\sum_{k=1}^{\infty}\int_{0}^{t}\int_{\mathbb{T}^{d}}\sigma^{k}(u_{n,\epsilon})\Phi_{n}(u_{n,\epsilon})dxdW_{s}^{k}\right|\nonumber\\
 & \leq C\mathbb{E}\left|\left(\int_{0}^{T}\sum_{k=1}^{\infty}\left(\int_{\mathbb{T}^{d}}|\sigma^{k}(u_{n,\epsilon})|\cdot|\Phi_{n}(u_{n,\epsilon})|dx\right)^{2}ds\right)^{\frac{1}{2}}\right|\nonumber \\
 & \leq C\mathbb{E}\Bigg|\bigg(\int_{0}^{T}\sum_{k=1}^{\infty}\left(\int_{\mathbb{T}^{d}}|\sigma^{k}(u_{n,\epsilon})|^{2}|\Phi_{n}(u_{n,\epsilon})|^{\frac{m-1}{m}}dx\right)\label{eq:secondestimateforstochastic} \\
 & \quad\cdot\left(\int_{\mathbb{T}^{d}}|\Phi_{n}(u_{n,\epsilon})|^{\frac{m+1}{m}}dx\right)ds\bigg)^{\frac{1}{2}}\Bigg|\nonumber \\
 & \leq C\mathbb{E}\left|\left(\int_{0}^{T}\left(\int_{\mathbb{T}^{d}}(1+|u_{n,\epsilon}|^{m+1})dx\right)^{2}ds\right)^{\frac{1}{2}}\right|\nonumber \\
 & \leq\frac{1}{2}\mathbb{E}\sup_{t\leq T}\left\Vert u_{n,\epsilon}\right\Vert _{L_{m+1}(\mathbb{T}^{d})}^{m+1}+C\left(1+\mathbb{E}\left\Vert u_{n,\epsilon}\right\Vert _{L_{m+1}(Q_{T})}^{m+1}\right).\nonumber 
\end{align}
Furthermore, since $u_{n,\epsilon}$ is non-negative and $\Phi_{n}$
is strictly increasing and odd, we have $\Phi_{n}(u_{n,\epsilon})\geq0$
almost everywhere in $Q_{T}$, almost surely, and 
\begin{equation}
-\frac{1}{\epsilon}\int_{0}^{T}\int_{\mathbb{T}^{d}}(u_{n,\epsilon}-S)^{+}\Phi_{n}(u_{n,\epsilon})dxds\leq0.\label{second estimate for penal}
\end{equation}
Combining with (\ref{eq:priori estimate2}) and (\ref{eq:second estimateforphi})-(\ref{second estimate for penal}),
we obtain the desired estimate.
\end{proof}
Using Galerkin approximation method as in \cite{dareiotis2020nonlinear},
we give the existence and uniqueness theorem, which extends \cite[Proposition 5.4]{dareiotis2020nonlinear}
to incorporate the barrier $S$.
\begin{thm}
\label{thm:exuni_approximation equation}Let Assumptions \ref{assu:assumption for phi}-\ref{assu:Assuption for barrier}
hold. Then, for all $n\in\mathbb{N}$ and $\epsilon>0$, $\Pi(\Phi_{n},F-G_{\epsilon}(\cdot,S),\xi_{n})$
admits a unique $L_{2}$-solution $u_{n,\epsilon}$. 
\end{thm}

\begin{proof}
Since we fix $n\in\mathbb{N}$ and $\epsilon>0$, we relabel $\bar{\Phi}:=\Phi_{n}$,
$\bar{\xi}:=\xi_{n}$ and $\bar{G}:=G_{\epsilon}$ in order to ease
the notation. Let $\{e_{l}\}_{l\in\mathbb{N}^{+}}\subset C^{\infty}(\mathbb{T}^{d})$
be an orthonormal basis of $L_{2}(\mathbb{T}^{d})$ consisting of
eigenvectors of $(I-\Delta)$. Define $H^{-i}(\mathbb{T}^{d})$ as
the dual of $H_{2,0}^{i}(\mathbb{T}^{d})$, equipped with the inner
product of $\langle\cdot,\cdot\rangle_{H^{-i}(\mathbb{T}^{d})}:=\langle(I-\Delta)^{-i/2}\cdot,(I-\Delta)^{-i/2}\cdot\rangle_{L_{2}(\mathbb{T}^{d})}$.
For any $l\in\mathbb{N}^{+}$, let $\varPi_{l}:H^{-1}(\mathbb{T}^{d})\mapsto V_{l}:=\mathrm{span}\{e_{1},\ldots,e_{l}\}$
be the projection operator, which means
\[
\varPi_{l}v:=\sum_{i=1}^{l}\ _{H^{-1}(\mathbb{T}^{d})}\langle v,e_{i}\rangle_{H^{1}(\mathbb{T}^{d})}e_{i},\quad\forall v\in H^{-1}(\mathbb{T}^{d}).
\]
The Galerkin approximation of $\Pi(\bar{\Phi},F-\bar{G}(\cdot,S),\bar{\xi})$
\begin{equation}
\begin{cases}
du_{l}=\varPi_{l}\left(\Delta\bar{\Phi}(u_{l})+F(t,x,u_{l})-\bar{G}(u_{l},S)\right)dt\\
\quad\quad\quad+{\displaystyle \sum_{k=1}^{\infty}}\varPi_{l}\sigma^{k}(u_{l})dW_{t}^{k},\quad(t,x)\in Q_{T};\\
u_{l}(0,x)=\varPi_{l}\bar{\xi}(x),\quad x\in\mathbb{T}^{d}
\end{cases}\label{eq:galerkin}
\end{equation}
is an equation on $V_{l}$, whose coefficients are locally Lipschitz
continuous and have a linear growth. Therefore, it admits a unique
solution $u_{l}$ satisfying
\[
u_{l}\in L_{2}(\Omega_{T};H^{1}(\mathbb{T}^{d}))\cap L_{2}(\Omega;C([0,T];L_{2}(\mathbb{T}^{d}))).
\]
Following the proof of Theorem \ref{thm:esimate for u_n,epsilon},
there exists a constant $C$ independent of $l\in\mathbb{N}^{+}$ such that
\begin{equation}
\mathbb{E}\int_{0}^{T}\left\Vert u_{l}\right\Vert _{H^{1}(\mathbb{T}^{d})}^{2}dt\leq C(1+\mathbb{E}\left\Vert \bar{\xi}\right\Vert _{L_{2}(\mathbb{T}^{d})}^{2}),\quad \text{and}\label{eq:est1 for ul}
\end{equation}
\begin{equation}
\mathbb{E}\sup_{t\leq T}\left\Vert u_{l}(t)\right\Vert _{L_{2}(\mathbb{T}^{d})}^{p}\leq C(1+\mathbb{E}\left\Vert \bar{\xi}\right\Vert _{L_{2}(\mathbb{T}^{d})}^{p}),\quad\forall p\in[2,\infty).\label{eq:est2 for ul}
\end{equation}
Moreover, we have almost surely, for all $t\in[0,T]$
\[
u_{l}(t)=J_{l}^{1}+J_{l}^{2}(t)+J_{l}^{3}(t),\quad\text{in}\ H^{-1}(\mathbb{T}^{d}),
\]
with
\[
J_{l}^{1}:=\varPi_{l}\bar{\xi},
\]
\[
J_{l}^{2}(t):=\int_{0}^{t}\varPi_{l}\left(\Delta\bar{\Phi}(u_{l})+F(s,\cdot,u_{l})-\bar{G}(u_{l},S)\right)ds,\quad\text{and}
\]
\[
J_{l}^{3}(t):=\sum_{k=1}^{\infty}\int_{0}^{t}\varPi_{l}\sigma^{k}(u_{l})dW_{s}^{k}.
\]
Using Sobolev's embedding theorem, inequality (\ref{eq:est1 for ul})
and the Lipschitz continuity of $F$ and $\bar{G}$, we have
\[
\sup_{l}\mathbb{E}\left\Vert J_{l}^{2}\right\Vert _{H_{4}^{\frac{1}{3}}([0,T];H^{-1}(\mathbb{T}^{d}))}^{2}\leq\sup_{l}\mathbb{E}\left\Vert J_{l}^{2}\right\Vert _{H^{1}([0,T];H^{-1}(\mathbb{T}^{d}))}^{2}<\infty.
\]
By \cite[Lemma 2.1]{flandoli1995martingale}, the linear growth of
$\sigma$ and (\ref{eq:est2 for ul}), we have
\[
\sup_{l}\mathbb{E}\left\Vert J_{l}^{3}\right\Vert _{H_{p}^{\alpha}([0,T];H^{-1}(\mathbb{T}^{d}))}^{p}<\infty,\quad\forall\alpha\in(0,\frac{1}{2}),\ p\in[2,\infty).
\]
Using these two estimates and (\ref{eq:est1 for ul}), we get
\[
\sup_{l}\mathbb{E}\left\Vert u_{l}\right\Vert _{H_{4}^{\frac{1}{3}}([0,T];H^{-1}(\mathbb{T}^{d}))\cap L_{2}([0,T];H^{1}(\mathbb{T}^{d}))}<\infty.
\]
Then, \cite[Theorem 2.1, Theorem 2.2]{flandoli1995martingale} yield the embedding
\begin{align*}
 & H_{4}^{\frac{1}{3}}([0,T];H^{-1}(\mathbb{T}^{d}))\cap L_{2}([0,T];H^{1}(\mathbb{T}^{d}))\\
 & \hookrightarrow\mathcal{X}:=L_{2}([0,T];L_{2}(\mathbb{T}^{d}))\cap C([0,T];H^{-2}(\mathbb{T}^{d}))
\end{align*}
is compact. Then, for any sequences $\{l_{q}\}_{q\in\mathbb{N}},\{\bar{l}_{q}\}_{q\in\mathbb{N}}\subset\mathbb{N}^{+}$,
the laws of $(u_{l_{q}},u_{\bar{l}_{q}})$ are tight on $\mathcal{X}\times\mathcal{X}$.
Define
\[
W(t):=\sum_{k=1}^{\infty}\frac{1}{\sqrt{2^{k}}}W_{t}^{k}\mathfrak{e}_{k},
\]
where $\{\mathfrak{e}_{k}\}_{k\in\mathbb{N}^{+}}$ is the standard
orthonormal basis of $l_{2}$. Moreover, from Assumption \ref{assu:Assuption for barrier},
it is easy to find $S\in L_{2}(\Omega;C[0,T])$. By Prokhorov's theorem,
there exists a (non-relabeled) subsequence $(u_{l_{q}},u_{\bar{l}_{q}})$
such that the laws of $\{(u_{l_{q}},u_{\bar{l}_{q}},W(\cdot),S(\cdot))\}_{q\in\mathbb{N}}$
on $\mathcal{Z:=\mathcal{X}}\times\mathcal{X}\times C([0,T];l_{2})\times C([0,T])$
are weakly convergent. By Skorokhod's representation theorem, there
exist $\mathcal{Z}$-valued random variables $(\hat{u},\check{u},\tilde{W}(\cdot),\tilde{S}(\cdot))$,
$\{(\hat{u}_{l_{q}},\check{u}_{\bar{l}_{q}},\tilde{W}_{q}(\cdot),\tilde{S}_{q}(\cdot))\}_{q\in\mathbb{N}}$
on a probability space $(\tilde{\Omega},\tilde{\mathcal{F}},\tilde{\mathbb{P}})$
such that in $\mathcal{Z}$, we have $\tilde{\mathbb{P}}$-almost
surely
\begin{equation}
(\hat{u}_{l_{q}},\check{u}_{\bar{l}_{q}},\tilde{W}_{q}(\cdot),\tilde{S}_{q}(\cdot))\xrightarrow{q\rightarrow\infty}(\hat{u},\check{u},\tilde{W}(\cdot),\tilde{S}(\cdot))\label{eq:limit for triple}
\end{equation}
and 
\begin{equation}
(\hat{u}_{l_{q}},\check{u}_{\bar{l}_{q}},\tilde{W}_{q}(\cdot),\tilde{S}_{q}(\cdot))\overset{d}{=}(u_{l_{q}},u_{\bar{l}_{q}},W(\cdot),S(\cdot)),\quad\forall q\in\mathbb{N}.\label{eq:equ distribution}
\end{equation}
Therefore for all $q\in\mathbb{N}$, we have
\[
\tilde{S}_{q}(t),\tilde{S}(t)\geq0,\quad\forall t\in[0,T],\quad\textrm{a.s.}\ \tilde{\omega}\in\tilde{\Omega}.
\]
Moreover, after passing to a non-relabeled subsequence $\{l_{q}\}_{q\in\mathbb{N}}$
and $\{\bar{l}_{q}\}_{q\in\mathbb{N}}$, we may assume that
\[
(\hat{u}_{l_{q}},\check{u}_{\bar{l}_{q}})\xrightarrow{q\rightarrow\infty}(\hat{u},\check{u}),\quad\textrm{a.s.}\ (\tilde{\omega},t,x)\in\tilde{\Omega}_{T}\times \mathbb{T}^{d}.
\]
Let $\{\tilde{\mathcal{F}}_{t}\}$ be the augmented filtration of
$\mathcal{G}_{t}:=\sigma(\hat{u}(s),\check{u}(s),\tilde{W}(s),\tilde{S}(s)|s\leq t)$,
and define $\tilde{W}_{q,t}^{k}:=\sqrt{2^{k}}\langle\tilde{W}_{q}(t),\mathfrak{e}_{k}\rangle_{l_{2}}$
and $\tilde{W}_{t}^{k}:=\sqrt{2^{k}}\langle\tilde{W}(t),\mathfrak{e}_{k}\rangle_{l_{2}}$.
As in the proof of \cite[Proposition 5.4]{dareiotis2020nonlinear},
it is easy to see that $\{\tilde{W}_{\cdot}^{k}\}_{k\in\mathbb{N}^{+}}$
are independent, standard and real-valued $\{\tilde{\mathcal{F}}_{t}\}$-adapted
Wiener processes. Moreover, Note that $\bar{G}(0,\tilde{S})=0$, $F(t,x,0)=0$
and $\{\tilde{S}_{q}\}_{q\in\mathbb{N}}$ is uniformly integrable.
Combining the Lipschitz continuity of $F$ and $\bar{G}$ with the
proof of \cite[Proposition 5.4]{dareiotis2020nonlinear}, we can prove
that both $\hat{u}$ and $\check{u}$ are $L_{2}$-solutions of $\ensuremath{\Pi(\bar{\Phi},F-\bar{G}(\cdot,\tilde{S}),\hat{\xi})}$,
where $\ensuremath{\hat{\xi}:=\hat{u}(0)}$. By Remark \ref{rem:L2equalentropy},
functions $\hat{u}$ and $\check{u}$ are also entropy solutions under
the Definition \ref{def:entropy solution without}.

Then, applying the Lipschitz continuity of $\bar{G}$ and $\bar{G}(0,\tilde{S})=0$
instead of the $L_{2}$ estimate of $\bar{G}$ in (\ref{eq:star property estimate of penalty}),
from Theorem \ref{thm:star property}, we know that both $\hat{u}$
and $\check{u}$ have the $(\star)$-property. By Theorem \ref{lem:L1}
(choose $G(t,x,r)=\tilde{G}(t,x,r)=\bar{G}(r,\tilde{S}(t))$) and
Gr\"onwall's inequality, we have $\hat{u}=\check{u}$ as in the proof
of Lemma \ref{lem:Lemma for comparison}. Based on \cite[Lemma 1]{gyongy1996existence},
we have that the initial sequence $\{u_{l}\}_{l=1}^{\infty}$ converges
in probability to some $u\in\mathcal{X}$. From this convergence and
the uniform estimates on $u_{l}$, one can pass to the limit in (\ref{eq:galerkin})
and obtain that $u$ is an $L_{2}$-solution of $\ensuremath{\Pi(\bar{\Phi},F-\bar{G}(\cdot,S),\bar{\xi})}$.

Note that the $L_{2}$-solution of $\ensuremath{\Pi(\bar{\Phi},F-\bar{G}(\cdot,S),\bar{\xi})}$
is also an entropy solution and has $(\star)$-property, then the
uniqueness of $u$ is acquired by applying Theorem \ref{lem:L1} and
Gr\"onwall's inequality as in the proof of Lemma \ref{lem:Lemma for comparison}.
\end{proof}
We have the following estimate for the penalty term $G_{\epsilon}$.
\begin{thm}
\label{thm:estimate for L1}Let Assumptions \ref{assu:assumption for phi}-\ref{assu:Assuption for barrier}
hold. Then, for all $n\in\mathbb{N}$ and $\epsilon>0$, there exists
a constant C independent of $n$ and $\epsilon$ such that
\[
\mathbb{E}\int_{0}^{t}\int_{\mathbb{T}^{d}}G_{\epsilon}(u_{n,\epsilon},S)dxds\leq C\left(1+\mathbb{E}\left\Vert \xi_{n}\right\Vert _{L_{2}(\mathbb{T}^{d})}^{2}\right).
\]
\end{thm}

\begin{proof}
Applying It\^{o}'s formula, we have
\begin{align*}
 & \int_{\mathbb{T}^{d}}\left(u_{n,\epsilon}(t,x)+1\right)^{2}dx\\
 & =\int_{\mathbb{T}^{d}}\left(\xi_{n}(x)+1\right)^{2}dx-2\int_{0}^{t}\langle\partial_{x_{i}}\Phi_{n}(u_{n,\epsilon}),\partial_{x_{i}}u_{n,\epsilon}\rangle_{L_{2}(\mathbb{T}^{d})}ds\\
 & \quad+2\int_{0}^{t}\int_{\mathbb{T}^{d}}\left(F(s,x,u_{n,\epsilon})-\frac{1}{\epsilon}(u_{n,\epsilon}-S)^{+}\right)\left(u_{n,\epsilon}+1\right)dxds\\
 & \quad+2\sum_{k=1}^{\infty}\int_{0}^{t}\langle\sigma^{k}(u_{n,\epsilon}),u_{n,\epsilon}+1\rangle_{L_{2}(\mathbb{T}^{d})}dW_{s}^{k}+\int_{0}^{t}\sum_{k=1}^{\infty}\left\Vert \sigma^{k}(u_{n,\epsilon})\right\Vert _{L_{2}(\mathbb{T}^{d})}^{2}ds.
\end{align*}
As in the proof of Theorem \ref{thm:esimate for u_n,epsilon}, with
Assumptions \ref{assu:assumption for phi}-\ref{assu:inhomogeneous function},
we have
\begin{align}
 & 2\int_{0}^{t}\int_{\mathbb{T}^{d}}\frac{1}{\epsilon}|(u_{n,\epsilon}-S)^{+}|^{2}dxds+2\int_{0}^{t}\int_{\mathbb{T}^{d}}\frac{1}{\epsilon}(u_{n,\epsilon}-S)^{+}dxds\nonumber\\
 & \leq C+\left\Vert \xi_{n}\right\Vert _{L_{2}(\mathbb{T}^{d})}^{2}+C\left\Vert u_{n.\epsilon}\right\Vert _{L_{2}(Q_{T})}^{2}\label{eq:u+1testfunction} \\
 & \quad+2\sum_{k=1}^{\infty}\int_{0}^{t}\langle\sigma^{k}(u_{n,\epsilon}),u_{n,\epsilon}+1\rangle_{L_{2}(\mathbb{T}^{d})}dW_{s}^{k}.\nonumber 
\end{align}
Since
\begin{align*}
 & \mathbb{E}\left|\sum_{k=1}^{\infty}\int_{0}^{t}\langle\sigma^{k}(u_{n,\epsilon}),u_{n,\epsilon}+1\rangle_{L_{2}(\mathbb{T}^{d})}dW_{s}^{k}\right|\\
 & \leq\mathbb{E}\left[\left|\sum_{k=1}^{\infty}\int_{0}^{t}\langle\sigma^{k}(u_{n,\epsilon}),u_{n,\epsilon}+1\rangle_{L_{2}(\mathbb{T}^{d})}^{2}ds\right|^{\frac{1}{2}}\right]\\
 & \leq\mathbb{E}\left[\left|\sum_{k=1}^{\infty}\int_{0}^{t}\left(\int_{\mathbb{T}^{d}}|\sigma^{k}(u_{n,\epsilon})|^{2}dx\right)\left(\int_{\mathbb{T}^{d}}|u_{n,\epsilon}+1|^{2}dx\right)ds\right|^{\frac{1}{2}}\right]\\
 & \leq C+C\mathbb{E}\left[\left|\int_{0}^{t}\left\Vert u_{n,\epsilon}\right\Vert _{L_{2}(\mathbb{T}^{d})}^{4}ds\right|^{\frac{1}{2}}\right]\\
 & \leq C+C\mathbb{E}\sup_{t\in[0,T]}\left\Vert u_{n,\epsilon}(t)\right\Vert _{L_{2}(\mathbb{T}^{d})}^{2}+C\mathbb{E}\left\Vert u_{n,\epsilon}\right\Vert _{L_{2}(Q_{T})}^{2},
\end{align*}
using Lemma \ref{Lem:nonnegative of u_n,e}, and inequalities (\ref{eq:priori estimate1})
and (\ref{eq:u+1testfunction}), we obtain the desired inequality.
\end{proof}
To obtain the $L_{2}$ estimate of $G_{\epsilon}(u_{n,\epsilon},S)$,
the specific form of stochastic differential equation in Assumption
\ref{assu:Assuption for barrier} is crucial, which gives that the
difference $u_{n,\epsilon}-S$ has bounded variation when $u_{n,\epsilon}=S$. 

Actually, the local martingale part will make it
fail to obtain better a priori estimate for the penalty term. However,
this term has no affect to the obstacle problem for backward equations
(cf. \cite{qiu2014quasi,yang2013dynkin}).
\begin{thm}
\label{thm:estimate for L2 for G}Let Assumptions \ref{assu:assumption for phi}-\ref{assu:Assuption for barrier}
hold. Then, for all $n\in\mathbb{N}$ and $\epsilon>0$, there exists
a constant C independent of $n$ and $\epsilon$ such that
\begin{align}
&\frac{1}{\epsilon}\mathbb{E}\left[\text{\ensuremath{\sup_{t\in[0,T]}}}\int_{\mathbb{T}^{d}}|(u_{n,\epsilon}-S)^{+}(t)|^{2}dx\right]+\frac{1}{\epsilon^{2}}\mathbb{E}\int_{0}^{T}\left\Vert (u_{n,\epsilon}-S)^{+}\right\Vert _{L_{2}(\mathbb{T}^{d})}^{2}dt\nonumber\\
&\leq C\left(1+\mathbb{E}\left\Vert \xi_{n}\right\Vert _{L_{2}(\mathbb{T}^{d})}^{2}\right).\label{eq:estimate for nu}
\end{align}
\end{thm}

\begin{proof}
We consider the equation
\[
\begin{cases}
d\big(u_{n,\epsilon}-S\big)=\big[\Delta\Phi_{n}(u_{n,\epsilon})+F(t,x,u_{n,\epsilon})-{\displaystyle \frac{1}{\epsilon}}(u_{n,\epsilon}-S)^{+}-h_{S}\big]dt\\
\quad\quad\quad\quad\quad\quad+{\displaystyle \sum_{k=1}^{\infty}}\big[\sigma^{k}(u_{n,\epsilon})-\sigma^{k}(S)\big]dW_{t}^{k},\quad(t,x)\in Q_{T};\\
u_{n,\epsilon}(0,x)-S(0)=\xi_{n}(x)-S_{0},\quad x\in\mathbb{T}^{d}.
\end{cases}
\]
Using It\^{o}'s formula (cf. the proof of \cite[Lemma 5.1]{yang2013dynkin}),
we have
\[
\frac{1}{\epsilon}\int_{\mathbb{T}^{d}}|(u_{n,\epsilon}-S)^{+}(t)|^{2}dx=\sum_{l=1}^{5}I_{l},
\]
where
\[
I_{1}:=\frac{1}{\epsilon}\int_{\mathbb{T}^{d}}|(\xi_{n}-S_{0})^{+}|^{2}dx,
\]
\[
I_{2}:=\frac{2}{\epsilon}\int_{0}^{t}\langle\Delta\Phi_{n}(u_{n,\epsilon}),(u_{n,\epsilon}-S)^{+}\rangle_{L_{2}(\mathbb{T}^{d})}ds,
\]
\[
I_{3}:=\frac{2}{\epsilon}\int_{0}^{t}\langle F(s,\cdot,u_{n,\epsilon})-\frac{1}{\epsilon}(u_{n,\epsilon}-S)^{+}-h_{S},(u_{n,\epsilon}-S)^{+}\rangle_{L_{2}(\mathbb{T}^{d})}ds,
\]
\[
I_{4}:=\frac{2}{\epsilon}\sum_{k=1}^{\infty}\int_{0}^{t}\langle\sigma^{k}(u_{n,\epsilon})-\sigma^{k}(S),(u_{n,\epsilon}-S)^{+}\rangle_{L_{2}(\mathbb{T}^{d})}dW_{s}^{k},
\]
\[
I_{5}:=\frac{1}{\epsilon}\sum_{k=1}^{\infty}\int_{0}^{t}\left\Vert \left(\sigma^{k}(u_{n,\epsilon})-\sigma^{k}(S)\right)\mathbf{1}_{\left\{ u_{n,\epsilon}\geq S\right\} }\right\Vert _{L_{2}(\mathbb{T}^{d})}^{2}ds.
\]
They are estimated below. Since $\xi_{n}\leq S_{0}$, we have $I_{1}\equiv0$.
In view of $\partial_{x_{i}}S\equiv0$ and Proposition \ref{prop:Assumptio for coef},
we have
\begin{align*}
I_{2}= & -\frac{2}{\epsilon}\int_{0}^{t}\langle\Phi_{n}^{\prime}(u_{n,\epsilon})\partial_{x_{i}}(u_{n,\epsilon}-S),\partial_{x_{i}}(u_{n,\epsilon}-S)\mathbf{1}_{\left\{ u_{n,\epsilon}\geq S\right\} }\rangle_{L_{2}(\mathbb{T}^{d})}ds\\
\leq & -\frac{8}{n^{2}\epsilon}\int_{0}^{t}\left\Vert \mathbf{1}_{\left\{ u_{n,\epsilon}\geq S\right\} }\partial_{x_{i}}(u_{n,\epsilon}-S)\right\Vert _{L_{2}(\mathbb{T}^{d})}^{2}ds.
\end{align*}

We have the following estimate for $I_{3}$
\begin{align*}
 & I_{3}+\frac{2}{\epsilon^{2}}\int_{0}^{t}\left\Vert (u_{n,\epsilon}-S)^{+}\right\Vert _{L_{2}(\mathbb{T}^{d})}^{2}ds\\
 & =\frac{2}{\epsilon}\int_{0}^{t}\langle F(s,\cdot,u_{n,\epsilon})-h_{S},(u_{n,\epsilon}-S)^{+}\rangle_{L_{2}(\mathbb{T}^{d})}ds\\
 & \leq\int_{0}^{t}\left\Vert F(s,\cdot,u_{n,\epsilon})-h_{S}\right\Vert _{L_{2}(\mathbb{T}^{d})}^{2}ds+\frac{1}{\epsilon^{2}}\int_{0}^{t}\left\Vert (u_{n,\epsilon}-S)^{+}\right\Vert _{L_{2}(\mathbb{T}^{d})}^{2}ds\\
 & \leq C+C\int_{0}^{t}\left\Vert u_{n,\epsilon}\right\Vert _{L_{2}(\mathbb{T}^{d})}^{2}ds+\frac{1}{\epsilon^{2}}\int_{0}^{t}\left\Vert (u_{n,\epsilon}-S)^{+}\right\Vert _{L_{2}(\mathbb{T}^{d})}^{2}ds.
\end{align*}

Now, we estimate $I_{4}$. Using Burkholder-Davis-Gundy inequality
and Assumption \ref{assu:assumption for sigma}, we have
\begin{align*}
 & \mathbb{E}\left[\ensuremath{\sup_{\tau\in[0,t]}}I_{4}\right]\\
 & \leq\frac{C}{\epsilon}\mathbb{E}\left[\left(\int_{0}^{t}\sum_{k=1}^{\infty}\langle\sigma^{k}(u_{n,\epsilon})-\sigma^{k}(S),(u_{n,\epsilon}-S)^{+}\rangle_{L_{2}(\mathbb{T}^{d})}^{2}ds\right)^{\frac{1}{2}}\right]\\
 & \leq\frac{C}{\epsilon}\mathbb{E}\Bigg[\Bigg(\int_{0}^{t}\sum_{k=1}^{\infty}\Big(\int_{\mathbb{T}^{d}}|\sigma^{k}(u_{n,\epsilon})-\sigma^{k}(S)|^{2}\mathbf{1}_{\left\{ u_{n,\epsilon}\geq S\right\} }dx\Big)\\
 & \quad\times\Big(\int_{\mathbb{T}^{d}}|(u_{n,\epsilon}-S)^{+}|^{2}dx\Big)ds\Bigg)^{\frac{1}{2}}\Bigg]\\
 & \leq\frac{C}{\epsilon}\mathbb{E}\Bigg[\Bigg(\int_{0}^{t}\Big(\int_{\mathbb{T}^{d}}(x^{2}+x^{1+2\kappa})\big|_{x=(u_{n,\epsilon}-S)^{+}}dx\Big)\cdot\Big(\int_{\mathbb{T}^{d}}|(u_{n,\epsilon}-S)^{+}|^{2}dx\Big)ds\Bigg)^{\frac{1}{2}}\Bigg].
\end{align*}
Using H\"older's inequality, we have
\begin{align*}
 & \mathbb{E}\left[\ensuremath{\sup_{\tau\in[0,t]}}I_{4}\right]\\
 & \leq\frac{C}{\epsilon}\mathbb{E}\Bigg[\Bigg(\Big(\text{\ensuremath{\sup_{\tau\in[0,t]}}}\int_{\mathbb{T}^{d}}|(u_{n,\epsilon}-S)^{+}(\tau)|^{2}dx\Big)\cdot\int_{0}^{t}\int_{\mathbb{T}^{d}}(x^{2}+x)\big|_{x=(u_{n,\epsilon}-S)^{+}}dxds\Bigg)^{\frac{1}{2}}\Bigg]\\
 & \leq\frac{1}{4\epsilon}\mathbb{E}\left[\text{\ensuremath{\sup_{\tau\in[0,t]}}}\int_{\mathbb{T}^{d}}|(u_{n,\epsilon}-S)^{+}(\tau)|^{2}dx\right]+\frac{C}{\epsilon}\left\Vert (u_{n,\epsilon}-S)^{+}\right\Vert _{L_{2}(\Omega_{T}\times\mathbb{T}^{d})}^{2}\\
 & \quad+\frac{C}{\epsilon}\left\Vert (u_{n,\epsilon}-S)^{+}\right\Vert _{L_{1}(\Omega_{T}\times\mathbb{T}^{d})}.
\end{align*}

In the same way, using Assumption \ref{assu:assumption for sigma},
we have
\begin{align*}
I_{5} & \leq\frac{C}{\epsilon}\Bigg(\int_{0}^{t}\left\Vert (u_{n,\epsilon}-S)^{+}\right\Vert _{L_{2}(\mathbb{T}^{d})}^{2}ds+\int_{0}^{t}\int_{\mathbb{T}^{d}}(u_{n,\epsilon}-S)^{+}dxds\Bigg).
\end{align*}

Combining the preceding five estimates, we have
\begin{align*}
 & \frac{3}{4\epsilon}\mathbb{E}\left[\text{\ensuremath{\sup_{t\in[0,T]}}}\int_{\mathbb{T}^{d}}\left|(u_{n,\epsilon}-S)^{+}(t)\right|^{2}dx\right]+\frac{1}{\epsilon^{2}}\left\Vert (u_{n,\epsilon}-S)^{+}\right\Vert _{L_{2}(\Omega_{T}\times\mathbb{T}^{d})}^{2}\\
 & \leq C+C\left\Vert u_{n,\epsilon}\right\Vert _{L_{2}(\Omega_{T}\times\mathbb{T}^{d})}^{2}+\frac{C}{\epsilon}\left\Vert (u_{n,\epsilon}-S)^{+}\right\Vert _{L_{2}(\Omega_{T}\times\mathbb{T}^{d})}^{2}\\
 & \quad+C\left\Vert \frac{1}{\epsilon}(u_{n,\epsilon}-S)^{+}\right\Vert _{L_{1}(\Omega_{T}\times\mathbb{T}^{d})}
\end{align*}
for a constant $C$ independent of $n$ and $\epsilon$. Using Theorems
\ref{thm:esimate for u_n,epsilon} and \ref{thm:estimate for L1},
we obtain the desired inequality.
\end{proof}

\section{$(\star)$-property}

We introduce the $(\star)$-property to give an estimate to stochastic
integral, which is a key step in the proof of $L_{1}^{+}$ estimate
between two entropy solutions. This method is also used in \cite{dareiotis2019entropy,dareiotis2020nonlinear}.

For simplicity, we denote the integral to time as $\int_{t}\cdot:=\int_{0}^{T}\cdot dt$
and $\int_{s}\cdot:=\int_{0}^{T}\cdot ds$, and denote the integral
to space as $\int_{x}\cdot:=\int_{\mathbb{T}^{d}}\cdot dx$ and $\int_{y}\cdot:=\int_{\mathbb{T}^{d}}\cdot dy$.
Let $g\in C^{\infty}(\mathbb{T}^{d}\times\mathbb{T}^{d})$ and $\varphi\in C_{c}^{\infty}((0,T))$.
For all $\theta>0$, we introduce
\[
\phi_{\theta}(t,x,s,y):=g(x,y)\rho_{\theta}(t-s)\varphi(\frac{t+s}{2}),\quad(t,x,s,y)\in Q_{T}\times Q_{T}.
\]
For all $\tilde{u}\in L_{m+1}(\Omega_{T};L_{m+1}(\mathbb{T}^{d}))$,
$h\in C^{\infty}(\mathbb{R})$ and $h^{\prime}\in C_{c}^{\infty}(\mathbb{R})$,
we further define
\[
H_{\theta}(t,x,a):=\sum_{k=1}^{\infty}\int_{0}^{T}\int_{y}h(\tilde{u}(s,y)-a)\sigma^{k}(y,\tilde{u}(s,y))\phi_{\theta}(t,x,s,y)dW_{s}^{k},\quad\forall a\in\mathbb{R}
\]
and 
\[
\mathcal{B}(u,\tilde{u},\theta):=-\sum_{k=1}^{\infty}\mathbb{E}\int_{t,x,s,y}\phi_{\theta}(t,x,s,y)\sigma^{k}(x,u(t,x))\sigma^{k}(y,\tilde{u}(s,y))h^{\prime}(\tilde{u}(s,y)-u(t,x)).
\]
It is easy to see that function $H_{\theta}$ is smooth in $(t,x,a)$.
Set $\mu:=\frac{3m+5}{4(m+1)}$. We have $\frac{m+3}{2(m+1)}<\mu<1$.
\begin{defn}
\label{def:star property}A function $u\in L_{m+1}(\Omega_{T}\times\mathbb{T}^{d})$
is said to have the $(\star)$-property if for all $(g,\varphi,\tilde{u},h,h^{\prime})\in C^{\infty}(\mathbb{T}^{d}\times\mathbb{T}^{d})\times C_{c}^{\infty}((0,T))\times L_{m+1}(\Omega_{T};L_{m+1}(\mathbb{T}^{d}))\times C^{\infty}(\mathbb{R})\times C_{c}^{\infty}(\mathbb{R})$,
and for all sufficiently small $\theta>0$, we have $H_{\theta}(\cdot,\cdot,u)\in L_{1}(\Omega_{T}\times\mathbb{T}^{d})$
and
\[
\mathbb{E}\int_{t,x}H_{\theta}(t,x,u(t,x))\leq C\theta^{1-\mu}+\mathcal{B}(u,\tilde{u},\theta)
\]
for a constant $C$ independent of $\theta$.
\end{defn}

\begin{rem}
\label{rem:simplify for F}Notice that $\varphi$ is supported in
$(0,T)$ and $\rho_{\theta}(t-\cdot)$ is supported in $[t-\theta,t]$.
For sufficiently small $\theta>0$, we have 
\[
H_{\theta}(t,x,a)=\sum_{k=1}^{\infty}\mathbf{1}_{t>\theta}\int_{t-\theta}^{t}\int_{y}h(\tilde{u}(s,y)-a)\sigma^{k}(y,\tilde{u}(s,y))\phi_{\theta}(t,x,s,y)dW_{s}^{k}.
\]
\end{rem}

The following three lemmas are introduced from \cite[Section 3]{dareiotis2019entropy}.
They are essential to the proofs to the $(\star)$-property of solution
$u$ and the $L_{1}^{+}$ estimate between two entropy solutions.
\begin{lem}
\label{lem:partial H}For all $\lambda\in(\frac{m+3}{2(m+1)},1)$,
$\bar{k}\in\mathbb{N}$ and sufficiently small $\theta\in(0,1)$,
we have 
\[
\mathbb{E}\left\Vert \partial_{a}H_{\theta}\right\Vert _{L_{\infty}([0,T];W_{m+1}^{\bar{k}}(\mathbb{T}^{d}\times\mathbb{R}))}^{m+1}\leq C\theta^{-\lambda(m+1)}\mathcal{N}_{m}(\tilde{u}),
\]
where
\[
\mathcal{N}_{m}(\tilde{u}):=\mathbb{E}\int_{0}^{T}\left(1+\left\Vert \tilde{u}(t)\right\Vert _{L_{\frac{m+1}{2}}(\mathbb{T}^{d})}^{m+1}+\left\Vert \tilde{u}(t)\right\Vert _{L_{2}(\mathbb{T}^{d})}^{m+1}\right)dt,
\]
and $C$ is a constant depending only on $N_{0}$, $N_{1}$, $\bar{k}$,
$d$, $T$, $\lambda$, $m$, and the functions $h$, $\varrho$,
$\varphi$, but not on $\theta$. In particular, we have
\[
\mathbb{E}\left\Vert \partial_{a}H_{\theta}\right\Vert _{L_{\infty}([0,T];W_{m+1}^{\bar{k}}(\mathbb{T}^{d}\times\mathbb{R}))}^{m+1}\leq C\theta^{-\lambda(m+1)}\left(1+\left\Vert \tilde{u}\right\Vert _{L_{m+1}(Q_{T})}^{m+1}\right).
\]
\end{lem}

\begin{lem}
\label{lem:appro for star}(i) Let $\{u_{n}\}_{n\in\mathbb{N}}$ be
a sequence bounded in $L_{m+1}(\Omega\times Q_{T})$, satisfying the
$(\star)$-property uniformly in $n$, which means the constant $C$
in Definition \ref{def:star property} is independent of $n$. If
$u_{n}$ converges to a function $u$ almost surely on $\Omega\times Q_{T}$,
then $u$ has the $(\star)$-property.

(ii) Let $u\in L_{2}(\Omega\times Q_{T})$. Then, for sufficiently
small $\theta\in(0,1)$, we have 
\[
\mathbb{E}\int_{t,x}H_{\theta}(t,x,u(t,x))=\lim_{\lambda\rightarrow0}\mathbb{E}\int_{t,x,a}H_{\theta}(t,x,a)\rho_{\lambda}(u(t,x)-a).
\]
\end{lem}

\begin{lem}
\label{lem:change for u}Let Assumption \ref{assu:assumption for phi}
holds and $u\in L_{1}(\Omega\times Q_{T})$. For some $\varepsilon\in(0,1)$,
let $\varrho:\mathbb{R}^{d}\mapsto\mathbb{R}$ be a non-negative function
integrating to one and supported on a ball of radius $\varepsilon$.
Then, we have
\[
\mathbb{E}\int_{t,x,y}|u(t,x)-u(t,y)|\varrho(x-y)\leq C\varepsilon^{\frac{2}{m+1}}(1+\mathbb{E}\left\Vert \nabla[\zeta](u)\right\Vert _{L_{1}(Q_{T})})
\]
for a constant $C$ independent of $d$, $K$ and $T$.
\end{lem}

Define $\varrho_{\varsigma}:=\rho_{\varsigma}^{\otimes d}$ for all
$\varsigma>0$. Now we prove that the solution $u_{n,\epsilon}$ has
the uniform $(\star)$-property.
\begin{thm}
\label{thm:star property}Let Assumptions \ref{assu:assumption for phi}-\ref{assu:Assuption for barrier}
hold. For any $n\in\mathbb{N}$ and $\epsilon>0$, let $u_{n,\epsilon}$
be the $L_{2}$-solution of $\Pi(\Phi_{n},F-G_{\epsilon}(\cdot,S),\xi_{n})$.
Then, $u_{n,\epsilon}$ has the $(\star)$-property. If in addition
$\left\Vert \xi\right\Vert _{L_{2}(\mathbb{T}^{d})}$ has moments
of order 4, then the constant $C$ in Definition \ref{def:star property}
is independent of $n$ and $\epsilon$.
\end{thm}

\begin{proof}
Fixed $\theta>0$ small enough so that Remark \ref{rem:simplify for F}
holds. We now apply the approximation method in the proof of \cite[Lemma 5.2]{dareiotis2019entropy}.
For a function $f\in L_{2}(\mathbb{T}^{d})$ and $\gamma>0$, let
$f^{(\gamma)}:=\varrho_{\gamma}*f$ be the mollification. Then, the
function $u_{n,\epsilon}^{(\gamma)}$ satisfies (pointwise) the equation
\begin{align*}
du_{n,\epsilon}^{(\gamma)} & =\bigg[\Big(\Delta\Phi_{n}(u_{n,\epsilon})+F(t,\cdot,u_{n,\epsilon})-\frac{1}{\epsilon}(u_{n,\epsilon}-S)^{+}\Big)^{(\gamma)}\bigg]dt+\sum_{k=1}^{\infty}\big(\sigma^{k}(u_{n,\epsilon})\big)^{(\gamma)}dW_{t}^{k}.
\end{align*}
Applying It\^{o}'s formula, we have
\[
\int_{t,x,a}H_{\theta}(t,x,a)\left(\rho_{\lambda}(u_{n,\epsilon}^{(\gamma)}(t,x)-a)-\rho_{\lambda}(u_{n,\epsilon}^{(\gamma)}(t-\theta,x)-a)\right)=\sum_{l=1}^{4}N_{\lambda,\gamma}^{(l)},
\]
where
\[
N_{\lambda,\gamma}^{(1)}:=\int_{t,x,a}H_{\theta}(t,x,a)\int_{t-\theta}^{t}\rho_{\lambda}^{\prime}(u_{n,\epsilon}^{(\gamma)}(s,x)-a)\Delta\big(\Phi_{n}(u_{n,\epsilon})\big)^{(\gamma)}ds,
\]
\[
N_{\lambda,\gamma}^{(2)}:=\int_{t,x,a}H_{\theta}(t,x,a)\sum_{k=1}^{\infty}\int_{t-\theta}^{t}\rho_{\lambda}^{\prime}(u_{n,\epsilon}^{(\gamma)}(s,x)-a)\big(\sigma^{k}(u_{n,\epsilon})\big)^{(\gamma)}dW_{s}^{k},
\]
\[
N_{\lambda,\gamma}^{(3)}:=\int_{t,x,a}H_{\theta}(t,x,a)\frac{1}{2}\int_{t-\theta}^{t}\rho_{\lambda}^{\prime\prime}(u_{n,\epsilon}^{(\gamma)}(s,x)-a)\sum_{k=1}^{\infty}\left|\big(\sigma^{k}(u_{n,\epsilon})\big)^{(\gamma)}\right|^{2}ds,
\]
\[
N_{\lambda,\gamma}^{(4)}:=\int_{t,x,a}H_{\theta}(t,x,a)\int_{t-\theta}^{t}\rho_{\lambda}^{\prime}(u_{n,\epsilon}^{(\gamma)}(s,x)-a)\big(F(s,\cdot,u_{n,\epsilon})-\frac{1}{\epsilon}(u_{n,\epsilon}-S)^{+}\big)^{(\gamma)}ds.
\]
For $N_{\lambda,\gamma}^{(4)}$, using integration by parts formula
(in $a$), we have
\begin{align*}
 & \mathbb{E}\left|N_{\lambda,\gamma}^{(4)}\right|\\
 & \leq\mathbb{E}\left|\int_{t,x,a}\partial_{a}H_{\theta}(t,x,a)\int_{t-\theta}^{t}\rho_{\lambda}(u_{n,\epsilon}^{(\gamma)}(s,x)-a)\big(F(s,\cdot,u_{n,\epsilon})-\frac{1}{\epsilon}(u_{n,\epsilon}-S)^{+}\big)^{(\gamma)}ds\right|\\
 & \leq N_{1}+N_{2},
\end{align*}
where
\[
N_{1}:=\mathbb{E}\left|\int_{t,x,a}\partial_{a}H_{\theta}(t,x,a)\int_{t-\theta}^{t}\rho_{\lambda}(u_{n,\epsilon}^{(\gamma)}(s,x)-a)\big(F(s,\cdot,u_{n,\epsilon})\big)^{(\gamma)}ds\right|
\]
and
\[
N_{2}:=\mathbb{E}\left|\int_{t,x,a}\partial_{a}H_{\theta}(t,x,a)\int_{t-\theta}^{t}\rho_{\lambda}(u_{n,\epsilon}^{(\gamma)}(s,x)-a)\big(\frac{1}{\epsilon}(u_{n,\epsilon}-S)^{+}\big)^{(\gamma)}ds\right|.
\]
Since
\[
\int_{x}f^{(\gamma)}(x)=\int_{x}\int_{y}\varrho_{\gamma}(x-y)=\int_{y}f(y)\int_{x}\varrho_{\gamma}(x-y)=2\int_{y}f(y),
\]
applying Assumptions \ref{assu:assumption for xi} and \ref{assu:inhomogeneous function},
Lemma \ref{lem:partial H} and Theorem \ref{thm:esimate for u_n,epsilon},
with \cite[Remark 3.2]{dareiotis2020nonlinear}, we have
\begin{align*}
N_{1} & \leq\mathbb{E}\left|\left\Vert \partial_{a}H_{\theta}\right\Vert _{L_{\infty}(Q_{T}\times\mathbb{R})}\int_{t,x}\int_{t-\theta}^{t}\big(F(s,\cdot,u_{n,\epsilon})\big)^{(\gamma)}ds\int_{a}(\rho_{\lambda}(u_{n,\epsilon}^{(\gamma)}(s,x)-a))\right|\\
 & \leq2\mathbb{E}\left|\left\Vert \partial_{a}H_{\theta}\right\Vert _{L_{\infty}(Q_{T}\times\mathbb{R})}\theta\int_{t,x}\big(F(t,\cdot,u_{n,\epsilon})\big)^{(\gamma)}\right|\\
 & \leq C\theta\left(\mathbb{E}\left\Vert \partial_{a}H_{\theta}\right\Vert _{L_{\infty}(Q_{T}\times\mathbb{R})}^{2}\right)^{1/2}\left(\mathbb{E}\left\Vert F(\cdot,\cdot,u_{n,\epsilon})\right\Vert _{L_{2}(Q_{T})}^{2}\right)^{1/2}\\
 & \leq C\theta\left(\mathbb{E}\left\Vert \partial_{a}H_{\theta}\right\Vert _{L_{\infty}(Q_{T}\times\mathbb{R})}^{2}\right)^{1/2}\left(1+\mathbb{E}\left\Vert u_{n,\epsilon}\right\Vert _{L_{2}(Q_{T})}^{2}\right)^{1/2}\leq C(n,\epsilon)\theta^{1-\mu}.
\end{align*}
Similarly, using Theorem \ref{thm:estimate for L2 for G}, we have
\begin{align}
N_{2} & \leq\mathbb{E}\left|\left\Vert \partial_{a}H_{\theta}\right\Vert _{L_{\infty}(Q_{T}\times\mathbb{R})}\int_{t,x}\int_{t-\theta}^{t}\big(\frac{1}{\epsilon}(u_{n,\epsilon}-S)^{+}\big)^{(\gamma)}ds\int_{a}(\rho_{\lambda}(u_{n,\epsilon}^{(\gamma)}(s,x)-a))\right|\nonumber\\
 & \leq\mathbb{E}\left|\left\Vert \partial_{a}H_{\theta}\right\Vert _{L_{\infty}(Q_{T}\times\mathbb{R})}\theta\int_{t,x}\big(\frac{1}{\epsilon}(u_{n,\epsilon}-S)^{+}\big)^{(\gamma)}\right|\label{eq:star property estimate of penalty} \\
 & \leq C\theta\left(\mathbb{E}\left\Vert \partial_{a}H_{\theta}\right\Vert _{L_{\infty}(Q_{T}\times\mathbb{R})}^{2}\right)^{1/2}\left(\mathbb{E}\left\Vert \frac{1}{\epsilon}(u_{n,\epsilon}-S)^{+}\right\Vert _{L_{2}(Q_{T})}^{2}\right)^{1/2}\leq C(n,\epsilon)\theta^{1-\mu}.\nonumber 
\end{align}
The estimates for $N_{\lambda,\gamma}^{(1)}$, $N_{\lambda,\gamma}^{(2)}$
and $N_{\lambda,\gamma}^{(3)}$ can be obtained as in the proof of
\cite[Lemma 5.2]{dareiotis2019entropy}. Combining these estimates
and following the proof of \cite[Lemma 5.2]{dareiotis2020nonlinear},
we have
\begin{align*}
 & \mathbb{E}\int_{t,x,a}H_{\theta}(t,x,u_{n,\epsilon}(t,x))\\
 & \leq\limsup_{\lambda\rightarrow0}\limsup_{\gamma\rightarrow0}\mathbb{E}\left(\left|N_{\lambda,\gamma}^{(1)}\right|+\left|N_{\lambda,\gamma}^{(3)}\right|\right)+\limsup_{\lambda\rightarrow0}\limsup_{\gamma\rightarrow0}\mathbb{E}\left|N_{\lambda,\gamma}^{(4)}\right|\\
 & \quad+\lim_{\lambda\rightarrow0}\lim_{\gamma\rightarrow0}\mathbb{E}N_{\lambda,\gamma}^{(2)}\leq C(n,\epsilon)\theta^{1-\mu}+\mathcal{B}(u_{n,\epsilon},\tilde{u},\theta).
\end{align*}
Moreover, if $\mathbb{E}\left\Vert \xi\right\Vert _{L_{2}(\mathbb{T}^{d})}^{4}<\infty$,
with Theorems \ref{thm:esimate for u_n,epsilon} and \ref{thm:estimate for L2 for G},
the constant $C(n,\epsilon)$ in the above can be selected to be independent
of $n$ and $\epsilon$.
\end{proof}

\section{$L_{1}^{+}$ estimate}

Note that $\varrho_{\varsigma}=\rho_{\varsigma}^{\otimes d}$.
\begin{lem}
\label{lem:Lemma for L1}Let $G(t,x,r)$ and $\tilde{G}(t,x,r)$ be
two functions, which are Lipschitz continuous in $r$, and satisfy
$G(\cdot,\cdot,0),\tilde{G}(\cdot,\cdot,0)\in\ensuremath{L_{2}(\Omega_{T};L_{2}(\mathbb{T}^{d}))}$.
Suppose that $u$ and $\tilde{u}$ are entropy solutions of $\Pi(\Phi,F-G,\xi)$
and $\Pi(\tilde{\Phi},F-\tilde{G},\tilde{\xi})$, respectively. Let
Assumptions \ref{assu:assumption for phi}-\ref{assu:inhomogeneous function}
hold for both $(\Phi,F,\sigma,\xi)$ and $(\tilde{\Phi},F,\sigma,\tilde{\xi})$.
If $u$ has the $(\star)$-property, then for every non-negative $\varphi\in C_{c}^{\infty}((0,T))$
such that 
\[
\left\Vert \varphi\right\Vert _{L_{\infty}(0,T)}\lor\left\Vert \partial_{t}\varphi\right\Vert _{L_{1}(0,T)}\leq1,
\]
and $\varsigma,\delta\in(0,1]$, $\lambda\in[0,1]$ and $\alpha\in(0,1\land(m/2))$,
we have
\begin{align}
 & -\mathbb{E}\int_{t,x,y}(u(t,x)-\tilde{u}(t,y))^{+}\varrho_{\varsigma}(x-y)\partial_{t}\varphi(t)\nonumber\\
 & \leq C\varsigma^{-2}\left(\mathbb{E}\left\Vert \mathbf{1}_{\{\left|u\right|\geq R_{\lambda}\}}(1+\left|u\right|)\right\Vert _{L_{m}(Q_{T})}^{m}+\mathbb{E}\left\Vert \mathbf{1}_{\{\left|\tilde{u}\right|\geq R_{\lambda}\}}(1+\left|\tilde{u}\right|)\right\Vert _{L_{m}(Q_{T})}^{m}\right)\nonumber \\
 & \quad+C\left(\delta^{2\kappa}+\varsigma^{\bar{\kappa}}+\varsigma^{-2}\lambda^{2}+\varsigma^{-2}\delta^{2\alpha}\right)\cdot\mathbb{E}\left(1+\left\Vert u\right\Vert _{L_{m+1}(Q_{T})}^{m+1}+\left\Vert \tilde{u}\right\Vert _{L_{m+1}(Q_{T})}^{m+1}\right)\label{eq:intermitineq} \\
 & \quad+C\mathbb{E}\int_{t,x,y}(u(t,x)-\tilde{u}(t,y))^{+}\varrho_{\varsigma}(x-y)\varphi(t)\nonumber \\
 & \quad+C\mathbb{E}\int_{t,x,y}\mathbf{1}_{\{\tilde{u}(t,y)\leq u(t,x)\}}\left(\tilde{G}(t,y,\tilde{u}(t,y))-G(t,x,u(t,x))\right)^{+}\varrho_{\varsigma}(x-y)\nonumber 
\end{align}
for a constant $C$ depending only on $N_{0}$, $K$, $d$ and $T$.
The parameter $R_{\lambda}$ is defined by
\[
R_{\lambda}:=\sup\left\{ R\in[0,\infty]:\left|\zeta(r)-\tilde{\zeta}(r)\right|\leq\lambda,\forall\left|r\right|<R\right\} .
\]
\end{lem}

\begin{proof}
For sufficiently small $\theta>0$, we introduce 
\[
\phi_{\theta,\varsigma}(t,x,s,y):=\rho_{\theta}(t-s)\varrho_{\varsigma}(x-y)\varphi(\frac{t+s}{2}),\quad \phi_{\varsigma}(t,x,y)=\varrho_{\varsigma}(x-y)\varphi(t).
\]
Furthermore, for each $\delta>0$, we define the function $\eta_{\delta}\in C^{2}(\mathbb{R})$
by
\[
\eta_{\delta}(0)=\eta_{\delta}^{\prime}(0)=0,\quad\eta_{\delta}^{\prime\prime}(r)=\rho_{\delta}(r).
\]
Thus, we have
\[
\left|\eta_{\delta}(r)-r^{+}\right|\leq\delta,\quad\mathrm{supp}\ \eta_{\delta}^{\prime\prime}\subset[0,\delta],\quad\int_{\mathbb{R}}\eta_{\delta}^{\prime\prime}(r)dr\leq2,\quad|\eta_{\delta}^{\prime\prime}|\leq2\delta^{-1}.
\]
Fix $(a,s,y)\in\mathbb{R}\times Q_{T}$. Since $u$ is the entropy
solution of $\Pi(\Phi,F-G,\xi)$, using the entropy inequality of
$u$ in Definition \ref{def:entropy solution without} with $\eta_{\delta}(r-a)$ and $\phi_{\theta,\varsigma}(\cdot,\cdot,s,y)$ instead of $\eta(r)$ and $\phi$, we have
\begin{align*}
 & -\int_{t,x}\eta_{\delta}(u-a)\partial_{t}\phi_{\theta,\varsigma}\\
 & \leq\int_{t,x}\llbracket \zeta^2 \eta_{\delta}^{\prime}(\cdot-a)\rrbracket (u) \Delta_{x}\phi_{\theta,\varsigma}+\int_{t,x}\eta_{\delta}^{\prime}(u-a)\left(F(t,x,u)-G(t,x,u)\right)\phi_{\theta,\varsigma}\\
 & \quad+\int_{t,x}\left(\frac{1}{2}\eta_{\delta}^{\prime\prime}(u-a)\sum_{k=1}^{\infty}|\sigma^{k}(u)|^{2}\phi_{\theta,\varsigma}-\eta_{\delta}^{\prime\prime}(u-a)|\nabla_{x}\left\llbracket \zeta\right\rrbracket (u)|^{2}\phi_{\theta,\varsigma}\right)\\
 & \quad+\sum_{k=1}^{\infty}\int_{0}^{T}\int_{x}\eta_{\delta}^{\prime}(u-a)\phi_{\theta,\varsigma}\sigma^{k}(u)dW_{t}^{k},
\end{align*}
where $u=u(t,x)$. Notice that all the expressions are continuous
in $(a,s,y)$. We take $a=\tilde{u}(s,y)$ by convolution and integrate
over $(s,y)\in Q_{T}$. By taking expectations, we have
\begin{align}
 & -\mathbb{E}\int_{t,x,s,y}\eta_{\delta}(u-\tilde{u})\partial_{t}\phi_{\theta,\varsigma}\nonumber\\
 & \leq\mathbb{E}\int_{t,x,s,y}\llbracket \zeta^2 \eta_{\delta}^{\prime}(\cdot-\tilde{u})\rrbracket (u) \Delta_{x}\phi_{\theta,\varsigma}+\mathbb{E}\int_{t,x,s,y}\eta_{\delta}^{\prime}(u-\tilde{u})\left(F(t,x,u)-G(t,x,u)\right)\phi_{\theta,\varsigma}\label{eq:entropyforu} \\
 & \quad+\mathbb{E}\int_{t,x,s,y}\left(\frac{1}{2}\eta_{\delta}^{\prime\prime}(u-\tilde{u})\sum_{k=1}^{\infty}|\sigma^{k}(u)|^{2}\phi_{\theta,\varsigma}-\eta_{\delta}^{\prime\prime}(u-\tilde{u})|\nabla_{x}\left\llbracket \zeta\right\rrbracket (u)|^{2}\phi_{\theta,\varsigma}\right)\nonumber \\
 & \quad-\mathbb{E}\int_{s,y}\left[\sum_{k=1}^{\infty}\int_{0}^{T}\int_{x}\eta_{\delta}^{\prime}(u-a)\phi_{\theta,\varsigma}\sigma^{k}(u)dW_{t}^{k}\right]_{a=\tilde{u}},\nonumber 
\end{align}
where $u=u(t,x)$ and $\tilde{u}=\tilde{u}(s,y)$. Similarly, for
each $(a,t,x)\in\mathbb{R}\times Q_{T}$ and entropy solution $\tilde{u}$,
we apply the entropy inequality of $\tilde{u}$ with $\eta(r):=\eta_{\delta}(a-r)$ and
$\phi(s,y):=\phi_{\theta,\varsigma}(t,x,s,y)$.
After substituting $a=u(t,x)$ by convolution, integrating over $(t,x)\in Q_{T}$
and taking expectations, we have
\begin{align*}
 & -\mathbb{E}\int_{t,x,s,y}\eta_{\delta}(u-\tilde{u})\partial_{s}\phi_{\theta,\varsigma}\\
 & \leq\mathbb{E}\int_{t,x,s,y}\llbracket \tilde{\zeta}^2 \eta_{\delta}^{\prime}(u-\cdot)\rrbracket (\tilde{u})\Delta_{y}\phi_{\theta,\varsigma}-\mathbb{E}\int_{t,x,s,y}\eta_{\delta}^{\prime}(u-\tilde{u})\big(F(s,y,\tilde{u})-\tilde{G}(s,y,\tilde{u})\big)\phi_{\theta,\varsigma}\\
 & \quad+\mathbb{E}\int_{t,x,s,y}\left(\frac{1}{2}\eta_{\delta}^{\prime\prime}(u-\tilde{u})\sum_{k=1}^{\infty}|\sigma^{k}(\tilde{u})|^{2}\phi_{\theta,\varsigma}-\eta_{\delta}^{\prime\prime}(u-\tilde{u})|\nabla_{y}\llbracket \tilde{\zeta}\rrbracket (\tilde{u})|^{2}\phi_{\theta,\varsigma}\right)\\
 & \quad-\mathbb{E}\int_{t,x}\left[\sum_{k=1}^{\infty}\int_{0}^{T}\int_{y}\eta_{\delta}^{\prime}(a-\tilde{u})\phi_{\theta,\varsigma}\sigma^{k}(\tilde{u})dW_{s}^{k}\right]_{a=u}.
\end{align*}
Adding them together, we have
\begin{equation}
-\mathbb{E}\int_{t,x,s,y}\eta_{\delta}(u-\tilde{u})\left(\partial_{t}\phi_{\theta,\varsigma}+\partial_{s}\phi_{\theta,\varsigma}\right)\leq\sum_{i=1}^{6}B_{i},\label{eq:inequalitieswiththeta}
\end{equation}
with
\[
B_{1}:=\mathbb{E}\int_{t,x,s,y}\llbracket \zeta^2 \eta_{\delta}^{\prime}(\cdot-\tilde{u})\rrbracket (u)\Delta_{x}\phi_{\theta,\varsigma}+\mathbb{E}\int_{t,x,s,y}\llbracket \tilde{\zeta}^2 \eta_{\delta}^{\prime}(u-\cdot)\rrbracket (\tilde{u})\Delta_{y}\phi_{\theta,\varsigma},
\]
\[
B_{2}:=\mathbb{E}\int_{t,x,s,y}\eta_{\delta}^{\prime}(u-\tilde{u})\Big[\big(F(t,x,u)-G(t,x,u)\big)-\big(F(s,y,\tilde{u})-\tilde{G}(s,y,\tilde{u})\big)\Big]\phi_{\theta,\varsigma},
\]
\[
B_{3}:=\mathbb{E}\int_{t,x,s,y}\left(\frac{1}{2}\eta_{\delta}^{\prime\prime}(u-\tilde{u})\sum_{k=1}^{\infty}|\sigma^{k}(u)|^{2}\phi_{\theta,\varsigma}-\eta_{\delta}^{\prime\prime}(u-\tilde{u})|\nabla_{x}\left\llbracket \zeta\right\rrbracket (u)|^{2}\phi_{\theta,\varsigma}\right),
\]
\[
B_{4}:=\mathbb{E}\int_{t,x,s,y}\left(\frac{1}{2}\eta_{\delta}^{\prime\prime}(u-\tilde{u})\sum_{k=1}^{\infty}|\sigma^{k}(\tilde{u})|^{2}\phi_{\theta,\varsigma}-\eta_{\delta}^{\prime\prime}(u-\tilde{u})|\nabla_{y}\llbracket \tilde{\zeta}\rrbracket (\tilde{u})|^{2}\phi_{\theta,\varsigma}\right),
\]
\[
B_{5}:=\mathbb{E}\int_{s,y}\left[\sum_{k=1}^{\infty}\int_{0}^{T}\int_{x}\eta_{\delta}^{\prime}(u-a)\phi_{\theta,\varsigma}\sigma^{k}(u)dW_{t}^{k}\right]_{a=\tilde{u}},
\]
\[
B_{6}:=\mathbb{E}\int_{t,x}\left[\sum_{k=1}^{\infty}\int_{0}^{T}\int_{y}\eta_{\delta}^{\prime}(a-\tilde{u})\phi_{\theta,\varsigma}\sigma^{k}(u)dW_{s}^{k}\right]_{a=u}.
\]
Since the integrand of the stochastic integral in $B_{5}$ vanish
on $[0,s]$, we have $B_{5}\equiv0$. 

For $B_{6}$, applying the $(\star)$-property of $u$ with $h(r):=-\eta_{\delta}^{\prime}(-r)$
and $g(x,y):=\varrho_{\varsigma}(x-y)$, we have
\[
B_{6}\leq C\theta^{1-\mu}-\sum_{k=1}^{\infty}\mathbb{E}\int_{t,x,s,y}\phi_{\theta,\varsigma}\sigma^{k}(u)\sigma^{k}(\tilde{u})\eta_{\delta}^{\prime\prime}(u-\tilde{u}),
\]
for a constant $C$ independent of $\theta$, and we have $\mu=\frac{3m+5}{4(m+1)}<1$.
Therefore, taking $\theta\rightarrow0^{+}$ as in the proof of \cite[Theorem 4.1]{dareiotis2019entropy},
we have
\begin{align}
 & -\mathbb{E}\int_{t,x,y}\eta_{\delta}(u-\tilde{u})\partial_{t}\phi_{\varsigma}\nonumber\\
 & \leq\mathbb{E}\int_{t,x,y}\llbracket \zeta^2 \eta_{\delta}^{\prime}(\cdot-\tilde{u})\rrbracket (u)\Delta_{x}\phi_{\varsigma}+\mathbb{E}\int_{t,x,y}\llbracket \tilde{\zeta}^2 \eta_{\delta}^{\prime}(u-\cdot)\rrbracket (\tilde{u})\Delta_{y}\phi_{\varsigma}\nonumber \\
 & \quad+\mathbb{E}\int_{t,x,y}\eta_{\delta}^{\prime}(u-\tilde{u})\Big[\big(F(t,x,u)-G(t,x,u)\big)-\big(F(t,y,\tilde{u})-\tilde{G}(t,y,\tilde{u})\big)\Big]\phi_{\varsigma}\label{eq:inequalitywitht} \\
 & \quad+\mathbb{E}\int_{t,x,y}\left(\frac{1}{2}\eta_{\delta}^{\prime\prime}(u-\tilde{u})\sum_{k=1}^{\infty}|\sigma^{k}(u)|^{2}\phi_{\varsigma}-\eta_{\delta}^{\prime\prime}(u-\tilde{u})|\nabla_{x}\left\llbracket \zeta\right\rrbracket (u)|^{2}\phi_{\varsigma}\right)\nonumber \\
 & \quad+\mathbb{E}\int_{t,x,y}\left(\frac{1}{2}\eta_{\delta}^{\prime\prime}(u-\tilde{u})\sum_{k=1}^{\infty}|\sigma^{k}(\tilde{u})|^{2}\phi_{\varsigma}-\eta_{\delta}^{\prime\prime}(u-\tilde{u})|\nabla_{y}\llbracket \tilde{\zeta}\rrbracket (\tilde{u})|^{2}\phi_{\varsigma}\right)\nonumber \\
 & \quad-\sum_{k=1}^{\infty}\mathbb{E}\int_{t,x,y}\phi_{\varsigma}\sigma^{k}(u)\sigma^{k}(\tilde{u})\eta_{\delta}^{\prime\prime}(u-\tilde{u}),\nonumber 
\end{align}
where $u=u(t,x)$ and $\tilde{u}=\tilde{u}(t,y)$. For the terms including
$\sigma^{k}$, using the property of $\eta_{\delta}$ and Assumption
\ref{assu:assumption for sigma}, we have
\begin{align}
 & \frac{1}{2}\mathbb{E}\int_{t,x,y}\eta_{\delta}^{\prime\prime}(u-\tilde{u})\sum_{k=1}^{\infty}\left(|\sigma^{k}(u)|^{2}-2\sigma^{k}(u)\sigma^{k}(\tilde{u})+|\sigma^{k}(\tilde{u})|^{2}\right)\phi_{\varsigma}\nonumber\\
 & \leq C\mathbb{E}\int_{t,x,y}\eta_{\delta}^{\prime\prime}(u-\tilde{u})\left(\sum_{k=1}^{\infty}|\sigma^{k}(u)-\sigma^{k}(\tilde{u})|^{2}\right)\phi_{\varsigma}\label{eq:sigmaL1} \\
 & \leq C\mathbb{E}\int_{t,x,y}\eta_{\delta}^{\prime\prime}(u-\tilde{u})\left|u-\tilde{u}\right|^{1+2\kappa}\phi_{\varsigma}\leq C\delta^{2\kappa}.\nonumber 
\end{align}
Furthermore, since $\partial_{x_{i}}\phi_{\varsigma}=-\partial_{y_{i}}\phi_{\varsigma}$ and $\partial_{x_i}\int_{0}^{\tilde{u}}\eta_{\delta}^{\prime}(r-\tilde{u})\zeta^2(r)dr=0$,
we have
\begin{align*}
N_{1} & :=\mathbb{E}\int_{t,x,y}\llbracket \zeta^2 \eta_{\delta}^{\prime}(\cdot-\tilde{u})\rrbracket (u)\Delta_{x}\phi_{\varsigma}\\
 & =-\mathbb{E}\int_{t,x,y}\mathbf{1}_{\{\tilde{u}\leq u\}}\partial_{x_{i}y_{i}}\phi_{\varsigma}\int_{\tilde{u}}^{u}\eta_{\delta}^{\prime}(r-\tilde{u})\zeta^{2}(r)dr\\
 & =-\mathbb{E}\int_{t,x,y}\mathbf{1}_{\{\tilde{u}\leq u\}}\partial_{x_{i}y_{i}}\phi_{\varsigma}\int_{\tilde{u}}^{u}\int_{\tilde{u}}^{r}\eta_{\delta}^{\prime\prime}(r-\tilde{r})\zeta^{2}(r)d\tilde{r}dr\\
 & =-\mathbb{E}\int_{t,x,y}\mathbf{1}_{\{\tilde{u}\leq u\}}\partial_{x_{i}y_{i}}\phi_{\varsigma}\int_{\tilde{u}}^{u}\int_{\tilde{u}}^{u}\mathbf{1}_{\{\tilde{r}\leq r\}}\eta_{\delta}^{\prime\prime}(r-\tilde{r})\zeta^{2}(r)d\tilde{r}dr.
\end{align*}
Similarly, we have
\begin{align*}
N_{2} & :=\mathbb{E}\int_{t,x,y}\llbracket \tilde{\zeta}^2 \eta_{\delta}^{\prime}(u-\cdot)\rrbracket (\tilde{u})\Delta_{y}\phi_{\varsigma}\\
 & =-\mathbb{E}\int_{t,x,y}\mathbf{1}_{\{\tilde{u}\leq u\}}\partial_{x_{i}y_{i}}\phi_{\varsigma}\int_{\tilde{u}}^{u}\eta_{\delta}^{\prime}(u-\tilde{r})\tilde{\zeta}^{2}(\tilde{r})d\tilde{r}\\
 & =-\mathbb{E}\int_{t,x,y}\mathbf{1}_{\{\tilde{u}\leq u\}}\partial_{x_{i}y_{i}}\phi_{\varsigma}\int_{\tilde{u}}^{u}\int_{\tilde{r}}^{u}\eta_{\delta}^{\prime\prime}(r-\tilde{r})\tilde{\zeta}^{2}(\tilde{r})drd\tilde{r}\\
 & =-\mathbb{E}\int_{t,x,y}\mathbf{1}_{\{\tilde{u}\leq u\}}\partial_{x_{i}y_{i}}\phi_{\varsigma}\int_{\tilde{u}}^{u}\int_{\tilde{u}}^{u}\mathbf{1}_{\{\tilde{r}\leq r\}}\eta_{\delta}^{\prime\prime}(r-\tilde{r})\tilde{\zeta}^{2}(\tilde{r})d\tilde{r}dr.
\end{align*}
Notice also that
\begin{align*}
N_{3} & :=-\mathbb{E}\int_{t,x,y}\eta_{\delta}^{\prime\prime}(u-\tilde{u})|\nabla_{x}\llbracket \zeta\rrbracket (u)|^{2}\phi_{\varsigma}-\mathbb{E}\int_{t,x,y}\eta_{\delta}^{\prime\prime}(u-\tilde{u})|\nabla_{y}\llbracket \tilde{\zeta}\rrbracket (\tilde{u})|^{2}\phi_{\varsigma}\\
 & \leq-2\mathbb{E}\int_{t,x,y}\eta_{\delta}^{\prime\prime}(u-\tilde{u})\nabla_{x}\llbracket \zeta\rrbracket (u)\cdot\nabla_{y}\llbracket \tilde{\zeta}\rrbracket (\tilde{u})\phi_{\varsigma}\\
 & =-2\mathbb{E}\int_{t,x,y}\mathbf{1}_{\{\tilde{u}\leq u\}}\phi_{\varsigma}\partial_{x_{i}}\left\llbracket \zeta\right\rrbracket (u)\partial_{y_{i}}\int_{u}^{\tilde{u}}\eta_{\delta}^{\prime\prime}(u-\tilde{r})\tilde{\zeta}(\tilde{r})d\tilde{r}\\
 & =2\mathbb{E}\int_{t,x,y}\mathbf{1}_{\{\tilde{u}\leq u\}}\partial_{y_{i}}\phi_{\varsigma}\partial_{x_{i}}\left\llbracket \zeta\right\rrbracket (u)\int_{u}^{\tilde{u}}\eta_{\delta}^{\prime\prime}(u-\tilde{r})\tilde{\zeta}(\tilde{r})d\tilde{r}.
\end{align*}
Applying \cite[Remark 3.1]{dareiotis2019entropy} with
\[
f(u):=\int_{u}^{\tilde{u}}\eta_{\delta}^{\prime\prime}(u-\tilde{r})\tilde{\zeta}(\tilde{r})d\tilde{r}
\]
and using $\partial_{x_{i}}\left\llbracket \zeta f\right\rrbracket (u)=f(u)\partial_{x_{i}}\left\llbracket \zeta\right\rrbracket (u)$,
we have
\begin{align*}
N_{3} & \leq-2\mathbb{E}\int_{t,x,y}\mathbf{1}_{\{\tilde{u}\leq u\}}\partial_{x_{i}y_{i}}\phi_{\varsigma}\int_{\tilde{u}}^{u}\int_{r}^{\tilde{u}}\eta_{\delta}^{\prime\prime}(r-\tilde{r})\tilde{\zeta}(\tilde{r})\zeta(r)d\tilde{r}dr\\
 & =2\mathbb{E}\int_{t,x,y}\mathbf{1}_{\{\tilde{u}\leq u\}}\partial_{x_{i}y_{i}}\phi_{\varsigma}\int_{\tilde{u}}^{u}\int_{\tilde{u}}^{u}\mathbf{1}_{\{\tilde{r}\leq r\}}\eta_{\delta}^{\prime\prime}(r-\tilde{r})\tilde{\zeta}(\tilde{r})\zeta(r)d\tilde{r}dr.
\end{align*}
Then, we have
\[
\sum_{i=1}^{3}N_{i}\leq\mathbb{E}\int_{t,x,y}\mathbf{1}_{\{\tilde{u}\leq u\}}|\partial_{x_{i}y_{i}}\phi_{\varsigma}|\int_{\tilde{u}}^{u}\int_{\tilde{u}}^{u}\mathbf{1}_{\{\tilde{r}\leq r\}}\eta_{\delta}^{\prime\prime}(r-\tilde{r})|\zeta(r)-\tilde{\zeta}(\tilde{r})|^{2}d\tilde{r}dr.
\]
With \cite[estimates (4.13)-(4.17)]{dareiotis2019entropy}, we have
\begin{align}
\sum_{i=1}^{3}N_{i} & \leq C\varsigma^{-2}(\delta^{2\alpha}+\lambda^{2})\mathbb{E}\left(1+\left\Vert u\right\Vert _{L_{m+1}(Q_{T})}^{m+1}+\left\Vert \tilde{u}\right\Vert _{L_{m+1}(Q_{T})}^{m+1}\right)\nonumber\\
 & \quad+C\varsigma^{-2}\mathbb{E}\left\Vert \mathbf{1}_{\{\left|u\right|\geq R_{\lambda}\}}(1+\left|u\right|)\right\Vert _{L_{m}(Q_{T})}^{m}\label{eq:phiL1} \\
 & \quad+C\varsigma^{-2}\mathbb{E}\left\Vert \mathbf{1}_{\{\left|\tilde{u}\right|\geq R_{\lambda}\}}(1+\left|\tilde{u}\right|)\right\Vert _{L_{m}(Q_{T})}^{m}.\nonumber 
\end{align}
For the other terms in the right hand side of (\ref{eq:inequalitywitht}),
from Assumption \ref{assu:inhomogeneous function}, we have
\begin{align}
 & \mathbb{E}\int_{t,x,y}\eta_{\delta}^{\prime}(u-\tilde{u})\left[F(t,x,u)-G(t,x,u)-F(t,y,\tilde{u})+\tilde{G}(t,y,\tilde{u})\right]\phi_{\varsigma}\nonumber\\
\text{} & \leq\mathbb{E}\int_{t,x,y}\eta_{\delta}^{\prime}(u-\tilde{u})\left(F(t,x,u)-F(t,y,u)\right)\phi_{\varsigma}\nonumber \\
 & \quad+\mathbb{E}\int_{t,x,y}\eta_{\delta}^{\prime}(u-\tilde{u})\left(F(t,y,u)-F(t,y,\tilde{u})\right)\phi_{\varsigma}\label{eq:FGdiffL1} \\
 & \quad+\mathbb{E}\int_{t,x,y}\eta_{\delta}^{\prime}(u-\tilde{u})\left(\tilde{G}(t,y,\tilde{u})-G(t,x,u)\right)\phi_{\varsigma}\nonumber \\
 & \leq C\varsigma^{\bar{\kappa}}+C\mathbb{E}\int_{t,x,y}(u-\tilde{u})^{+}\varrho_{\varsigma}(x-y)\varphi(t)\nonumber \\
 & \quad+C\mathbb{E}\int_{t,x,y}\mathbf{1}_{\{u>\tilde{u}\}}\left(\tilde{G}(t,y,\tilde{u})-G(t,x,u)\right)\varrho_{\varsigma}(x-y)\varphi(t).\nonumber 
\end{align}
Combining inequalities (\ref{eq:inequalitywitht})-(\ref{eq:FGdiffL1})
with
\[
|\mathbb{E}\int_{t,x,y}\eta_{\delta}(u-\tilde{u})\partial_{t}\phi_{\varsigma}-\mathbb{E}\int_{t,x,y}(u-\tilde{u})^{+}\partial_{t}\phi_{\varsigma}|\leq C\delta,
\]
we obtain the desired inequality.
\end{proof}
\begin{lem}
\label{lem:appro initial}Let Assumptions \ref{assu:assumption for phi}-\ref{assu:inhomogeneous function}
hold for $(\Phi,F,\sigma,\xi)$. Let $G(t,x,r)$ be a function satisfying
$G(\cdot,\cdot,0)\in\ensuremath{L_{2}(\Omega_{T};L_{2}(\mathbb{T}^{d}))}$,
and be Lipschitz continuous in $r$ with Lipschitz constant $\bar{K}$.
If $u$ is an entropy solution of $\Pi(\Phi,F-G,\xi)$, we have
\[
\lim_{\tau\rightarrow0^{+}}\frac{1}{\tau}\mathbb{E}\int_{0}^{\tau}\int_{x}\left|u(t,x)-\xi(x)\right|^{2}dt=0.
\]
\end{lem}

The proof of Lemma \ref{lem:appro initial} is similar to that of
\cite[Lemma 3.2]{dareiotis2019entropy} under the Lipschitz continuity
of $F$ and $G$. Therefore, we omit the proof here.
\begin{lem}
\label{lem:L1}Let Assumptions \ref{assu:assumption for phi}-\ref{assu:inhomogeneous function}
hold for both $(\Phi,F,\sigma,\xi)$ and $(\tilde{\Phi},F,\sigma,\tilde{\xi})$.
Let $G(t,x,r)$ and $\tilde{G}(t,x,r)$ be two functions, which are
Lipschitz continuous in $r$ with Lipschitz constant $\bar{K}$, and
satisfy $G(\cdot,\cdot,0),\tilde{G}(\cdot,\cdot,0)\in\ensuremath{L_{2}(\Omega_{T};L_{2}(\mathbb{T}^{d}))}$.
Suppose that $u$ and $\tilde{u}$ are entropy solutions of $\Pi(\Phi,F-G,\xi)$
and $\Pi(\tilde{\Phi},F-\tilde{G},\tilde{\xi})$, respectively. If
$u$ has the $(\star)$-property, then the following two assertions
are true:

(i) if furthermore $\Phi=\tilde{\Phi}$, then for all $\varsigma,\delta\in(0,1]$
and $\alpha\in(0,1\land(m/2))$, we have
\begin{align*}
 & \mathbb{E}\int_{x,y}(u(\tau,x)-\tilde{u}(\tau,y))^{+}\varrho_{\varsigma}(x-y)\\
 & \leq\mathbb{E}\int_{x,y}(\xi(x)-\tilde{\xi}(y))^{+}\varrho_{\varsigma}(x-y)\\
 & \quad+C\mathfrak{C}(\varsigma,\delta)\mathbb{E}\left(1+\left\Vert u\right\Vert _{L_{m+1}(Q_{T})}^{m+1}+\left\Vert \tilde{u}\right\Vert _{L_{m+1}(Q_{T})}^{m+1}\right)\\
 & \quad+C\mathbb{E}\int_{0}^{\tau}\int_{x,y}(u(t,x)-\tilde{u}(t,y))^{+}\varrho_{\varsigma}(x-y)dt\\
 & \quad+C\mathbb{E}\int_{0}^{\tau}\int_{x,y}\mathbf{1}_{\{u(t,x)>\tilde{u}(t,y)\}}(\tilde{G}(t,y,\tilde{u}(t,y))-G(t,x,u(t,x)))^{+}\varrho_{\varsigma}(x-y)dt,
\end{align*}
where
\[
\mathfrak{C}(\varsigma,\delta)=\delta^{2\kappa}+\varsigma^{\bar{\kappa}}+\varsigma^{-2}\delta^{2\alpha}.
\]

(ii) for all $\varsigma,\delta\in(0,1]$, $\lambda\in[0,1]$ and $\alpha\in(0,1\land(m/2))$,
we have
\begin{align*}
 & \mathbb{E}\int_{t,x}(u(t,x)-\tilde{u}(t,x))^{+}\\
 & \leq C\mathbb{E}\int_{x}(\xi(x)-\tilde{\xi}(x))^{+}+C\sup_{|h|\leq\varsigma}\mathbb{E}\left\Vert \tilde{\xi}(\cdot)-\tilde{\xi}(\cdot+h)\right\Vert _{L_{1}(\mathbb{T}^{d})}\\
 & \quad+C\varsigma^{-2}\left(\mathbb{E}\left\Vert \mathbf{1}_{\{\left|u\right|\geq R_{\lambda}\}}(1+\left|u\right|)\right\Vert _{L_{m}(Q_{T})}^{m}+\mathbb{E}\left\Vert \mathbf{1}_{\{\left|\tilde{u}\right|\geq R_{\lambda}\}}(1+\left|\tilde{u}\right|)\right\Vert _{L_{m}(Q_{T})}^{m}\right)\\
 & \quad+C\mathfrak{C}(\varsigma,\delta,\lambda)\mathbb{E}\left(1+\left\Vert u\right\Vert _{L_{m+1}(Q_{T})}^{m+1}+\left\Vert \tilde{u}\right\Vert _{L_{m+1}(Q_{T})}^{m+1}\right)\\
 & \quad+C\varsigma^{\frac{2}{m+1}}\left(1+\mathbb{E}\left\Vert \nabla\left\llbracket \zeta\right\rrbracket (u)\right\Vert _{L_{1}(Q_{T})}\right)\\
 & \quad+C\mathbb{E}\int_{t,x,y}\mathbf{1}_{\{u(t,x)>\tilde{u}(t,y)\}}(\tilde{G}(t,y,\tilde{u}(t,y))-G(t,x,u(t,x)))^{+}\varrho_{\varsigma}(x-y),
\end{align*}
where
\[
\mathfrak{C}(\varsigma,\delta,\lambda)=\delta^{2\kappa}+\varsigma^{\bar{\kappa}}+\varsigma^{-2}\lambda^{2}+\varsigma^{-2}\delta^{2\alpha},
\]
\[
R_{\lambda}=\sup\{R\in[0,\infty]:|\zeta(r)-\tilde{\zeta}(r)|\leq\lambda,\forall|r|<R\},
\]
and the constant $C$ depends only on $N_{0}$, $K$, $d$, $T$ and
$\phi$. 
\end{lem}

\begin{proof}
Let $s,\tau\in(0,T)$ with $s<\tau$, be Lebesgue points of the function
\[
t\mapsto\mathbb{E}\int_{x,y}(u(t,x)-\tilde{u}(t,y))^{+}\varrho_{\varsigma}(x-y).
\]
Fix a constant $\gamma\in(0,\max\{\tau-s,T-\tau\})$. We choose a
sequence of functions $\{\varphi_{n}\}_{n\in\mathbb{N}}$ satisfying
$\varphi_{n}\in C_{c}^{\infty}((0,T))$ and $\left\Vert \varphi_{n}\right\Vert _{L_{\infty}(0,T)}\lor\left\Vert \partial_{t}\varphi_{n}\right\Vert _{L_{1}(0,T)}\leq1$,
such that
\[
\lim_{n\rightarrow\infty}\left\Vert \varphi_{n}-V_{\gamma}\right\Vert _{H_{0}^{1}((0,T))}=0,
\]
where $V_{\gamma}:[0,T]\rightarrow\mathbb{R}$ satisfies $V_{\gamma}(0)=0$
and $V_{\gamma}^{\prime}=\gamma^{-1}\mathbf{1}_{[s,s+\gamma]}-\gamma^{-1}\mathbf{1}_{[\tau,\tau+\gamma]}$.
Substituting $\varphi$ with $\varphi_{n}$ in (\ref{eq:intermitineq})
and taking the limit $n\rightarrow\infty$, we have
\begin{align*}
 & \frac{1}{\gamma}\mathbb{E}\int_{\tau}^{\tau+\gamma}\int_{x,y}(u(t,x)-\tilde{u}(t,y))^{+}\varrho_{\varsigma}(x-y)dt\\
 & \leq C\varsigma^{-2}\left(\mathbb{E}\left\Vert \mathbf{1}_{\{\left|u\right|\geq R_{\lambda}\}}(1+\left|u\right|)\right\Vert _{L_{m}(Q_{T})}^{m}+\mathbb{E}\left\Vert \mathbf{1}_{\{\left|\tilde{u}\right|\geq R_{\lambda}\}}(1+\left|\tilde{u}\right|)\right\Vert _{L_{m}(Q_{T})}^{m}\right)\\
 & \quad+C\left(\delta^{2\kappa}+\varsigma^{\bar{\kappa}}+\varsigma^{-2}\lambda^{2}+\varsigma^{-2}\delta^{2\alpha}\right)\cdot\mathbb{E}\left(1+\left\Vert u\right\Vert _{L_{m+1}(Q_{T})}^{m+1}+\left\Vert \tilde{u}\right\Vert _{L_{m+1}(Q_{T})}^{m+1}\right)\\
 & \quad+C\mathbb{E}\int_{0}^{\tau+\gamma}\int_{x,y}(u(t,x)-\tilde{u}(t,y))^{+}\varrho_{\varsigma}(x-y)dt\\
 & \quad+C\mathbb{E}\int_{0}^{\tau+\gamma}\int_{x,y}\mathbf{1}_{\{u(t,x)>\tilde{u}(t,y)\}}(\tilde{G}(t,y,\tilde{u}(t,y))-G(t,x,u(t,x)))^{+}\varrho_{\varsigma}(x-y)dt\\
 & \quad+\frac{1}{\gamma}\mathbb{E}\int_{s}^{s+\gamma}\int_{x,y}(u(t,x)-\tilde{u}(t,y))^{+}\varrho_{\varsigma}(x-y)dt\\
 & =:M(\gamma)+\frac{1}{\gamma}\mathbb{E}\int_{s}^{s+\gamma}\int_{x,y}(u(t,x)-\tilde{u}(t,y))^{+}\varrho_{\varsigma}(x-y)dt.
\end{align*}
Let $\gamma\rightarrow0^{+}$, we have
\[
\mathbb{E}\int_{x,y}(u(\tau,x)-\tilde{u}(\tau,y))^{+}\varrho_{\varsigma}(x-y)\leq M(0)+\mathbb{E}\int_{x,y}(u(s,x)-\tilde{u}(s,y))^{+}\varrho_{\varsigma}(x-y)
\] 
holds for almost all $s\in(0,\tau)$. Then, for each $\tilde{\gamma}\in(0,\tau),$
by averaging over $s\in(0,\tilde{\gamma})$, we have
\begin{align*}
 & \mathbb{E}\int_{x,y}(u(\tau,x)-\tilde{u}(\tau,y))^{+}\varrho_{\varsigma}(x-y)\\
 & \leq M(0)+\frac{1}{\tilde{\gamma}}\mathbb{E}\int_{0}^{\tilde{\gamma}}\int_{x,y}(u(s,x)-\tilde{u}(s,y))^{+}\varrho_{\varsigma}(x-y)ds.
\end{align*}
Taking the limit $\tilde{\gamma}\rightarrow0^{+}$ and using Lemma
\ref{lem:appro initial}, we have
\begin{equation}
\mathbb{E}\int_{x,y}(u(\tau,x)-\tilde{u}(\tau,y))^{+}\varrho_{\varsigma}(x-y)\leq M(0)+\mathbb{E}\int_{x,y}(\xi(x)-\tilde{\xi}(y))^{+}\varrho_{\varsigma}(x-y).\label{eq:ineq_givenphis} 
\end{equation}
Taking $\lambda=0$ and $R_{\lambda}=\infty$, we obtain the desired
inequality in (i).

For (ii), we fixed $s_{1}\in(0,T]$. By integrating inequality (\ref{eq:ineq_givenphis})
over $\tau\in(0,s_{1})$, we have
\begin{align*}
 & \mathbb{E}\int_{0}^{s_{1}}\int_{x,y}(u(\tau,x)-\tilde{u}(\tau,y))^{+}\varrho_{\varsigma}(x-y)d\tau\\
 & \leq T\mathbb{E}\int_{x}(\xi(x)-\tilde{\xi}(x))^{+}+T\sup_{|h|\leq\varsigma}\mathbb{E}\left\Vert \tilde{\xi}(\cdot)-\tilde{\xi}(\cdot+h)\right\Vert _{L_{1}(\mathbb{T}^{d})}\\
 & \quad+C\varsigma^{-2}\Bigg(\mathbb{E}\left\Vert \mathbf{1}_{\{\left|u\right|\geq R_{\lambda}\}}(1+\left|u\right|)\right\Vert _{L_{m}(Q_{T})}^{m}+\mathbb{E}\left\Vert \mathbf{1}_{\{\left|\tilde{u}\right|\geq R_{\lambda}\}}(1+\left|\tilde{u}\right|)\right\Vert _{L_{m}(Q_{T})}^{m}\Bigg)\\
 & \quad+C\left(\delta^{2\kappa}+\varsigma^{\bar{\kappa}}+\varsigma^{-2}\lambda^{2}+\varsigma^{-2}\delta^{2\alpha}\right)\cdot\mathbb{E}\Big(1+\left\Vert u\right\Vert _{L_{m+1}(Q_{T})}^{m+1}+\left\Vert \tilde{u}\right\Vert _{L_{m+1}(Q_{T})}^{m+1}\Big)\\
 & \quad+C\mathbb{E}\int_{0}^{s_{1}}\int_{0}^{\tau}\int_{x,y}(u(t,x)-\tilde{u}(t,y))^{+}\varrho_{\varsigma}(x-y)dtd\tau\\
 & \quad+C\mathbb{E}\int_{t,x,y}\mathbf{1}_{\{u(t,x)>\tilde{u}(t,y)\}}(\tilde{G}(t,y,\tilde{u}(t,y))-G(t,x,u(t,x)))^{+}\varrho_{\varsigma}(x-y),
\end{align*}
Moreover, Using Lemma \ref{lem:change for u} and Gr\"onwall's inequality,
we have Assertion (ii).
\end{proof}

\section{Existence of solution\label{sec:Existence-of-solution}}

We have already obtained a priori estimates and properties of $L_{2}$-solution
$u_{n,\epsilon}$ to $\Pi(\Phi_{n},F-G_{\epsilon}(\cdot,S),\xi_{n})$,
and gotten $L_{1}^{+}$ estimate of two different entropy solutions
based on $(\star)$-property. Applying these results, we now prove
the existence of entropy solution $(u,\nu)$ of the obstacle problem
$\Pi_{S}(\Phi,F,\xi)$ in two steps: Firstly, we take the limit $n\rightarrow\infty$
to prove the existence and comparison theorem of the entropy solution
$u_{\epsilon}$ of $\Pi(\Phi,F-G_{\epsilon}(\cdot,S),\xi)$. Then,
these results indicate the existence of the entropy solution of the
obstacle problem $\Pi_{S}(\Phi,F,\xi)$ when $\epsilon\rightarrow0^{+}$.

Fix $\epsilon>0$, for any $n,n^{\prime}\in\mathbb{N}$, suppose that
$u_{n,\epsilon}$ and $u_{n^{\prime},\epsilon}$ are $L_{2}$-solutions
of $\Pi(\Phi_{n},F-G_{\epsilon}(\cdot,S),\xi_{n})$ and $\Pi(\Phi_{n^{\prime}},F-G_{\epsilon}(\cdot,S),\xi_{n^{\prime}})$,
respectively. Then, Remark \ref{rem:L2equalentropy} shows that they
are also entropy solutions of the corresponding equations. Using Theorem
\ref{thm:star property}, we know that both $u_{n,\epsilon}$ and
$u_{n^{\prime},\epsilon}$ have the $(\star)$-property, which is
uniform in $n$ and $\epsilon$ if $\mathbb{E}\left\Vert \xi\right\Vert _{L_{2}(\mathbb{T}^{d})}^{4}<\infty$. 
\begin{thm}
\label{The:ex and uni for u_e}Let Assumptions \ref{assu:assumption for phi}-\ref{assu:Assuption for barrier}
hold. Then, for fixed $\epsilon>0$, the equation $\Pi(\Phi,F-G_{\epsilon}(\cdot,S),\xi)$
has an entropy solution $u_{\epsilon}$. Moreover, there exists a
constant $C$ independent of $\epsilon$ such that
\begin{align}
&\mathbb{E}\sup_{t\leq T}\left\Vert u_{\epsilon}(t)\right\Vert _{L_{2}(\mathbb{T}^{d})}^{p}+\mathbb{E}\left\Vert \nabla\left\llbracket \zeta\right\rrbracket (u_{\epsilon})\right\Vert _{L_{2}(Q_{T})}^{p}+(\frac{1}{\epsilon})^{p/2}\mathbb{E}\left\Vert (u_{\epsilon}-S)^{+}\right\Vert _{L_{2}(Q_{T})}^{p}\label{eq:u_e priori estimate1}\\
&\leq C\left(1+\mathbb{E}\left\Vert \xi\right\Vert _{L_{2}(\mathbb{T}^{d})}^{p}\right),\nonumber
\end{align}
\begin{equation}
\mathbb{E}\sup_{t\leq T}\left\Vert u_{\epsilon}(t)\right\Vert _{L_{m+1}(\mathbb{T}^{d})}^{m+1}\leq C\left(1+\mathbb{E}\left\Vert \xi\right\Vert _{L_{m+1}(\mathbb{T}^{d})}^{m+1}\right),\quad\text{and}\label{eq:u_e priori estimate2}
\end{equation}
\begin{equation}
\frac{1}{\epsilon^{2}}\mathbb{E}\int_{0}^{T}\left\Vert (u_{\epsilon}-S)^{+}(t)\right\Vert _{L_{2}(\mathbb{T}^{d})}^{2}dt\leq C\left(1+\mathbb{E}\left\Vert \xi\right\Vert _{L_{2}(\mathbb{T}^{d})}^{2}\right).\label{eq:nu priori estimate}
\end{equation}
\end{thm}

\begin{proof}
We take
\[
G(t,x,r)=\tilde{G}(t,x,r):=G_{\epsilon}(r,S(t))
\]
and $u=u_{n,\epsilon}$, $\tilde{u}=u_{n^{\prime},\epsilon}$ in Lemma
\ref{lem:L1}(ii) with fixed $\epsilon>0$. Because $n$ and $n^{\prime}$
have the same status, we can obtain a same inequality with swapping
$n$ and $n^{\prime}$. Adding them together, we have
\begin{align*}
 & \mathbb{E}\int_{t,x}\left|u_{n,\epsilon}(\tau,x)-u_{n^{\prime},\epsilon}(\tau,x)\right|\\
 & \leq C\mathbb{E}\int_{x}\left|\xi_{n}(x)-\xi_{n^{\prime}}(x)\right|+C\sup_{|h|\leq\varsigma}\mathbb{E}\left\Vert \xi_{n^{\prime}}(\cdot)-\xi_{n^{\prime}}(\cdot+h)\right\Vert _{L_{1}(\mathbb{T}^{d})}\\
 & \quad+C\sup_{|h|\leq\varsigma}\mathbb{E}\left\Vert \xi_{n}(\cdot)-\xi_{n}(\cdot+h)\right\Vert _{L_{1}(\mathbb{T}^{d})}\\
 & \quad+C\varsigma^{-2}\Bigg(\mathbb{E}\left\Vert \mathbf{1}_{\{\left|u_{n,\epsilon}\right|\geq R_{\lambda}\}}(1+\left|u_{n,\epsilon}\right|)\right\Vert _{L_{m}(Q_{T})}^{m}\\
 & \quad+\mathbb{E}\left\Vert \mathbf{1}_{\{\left|u_{n^{\prime},\epsilon}\right|\geq R_{\lambda}\}}(1+\left|u_{n^{\prime},\epsilon}\right|)\right\Vert _{L_{m}(Q_{T})}^{m}\Bigg)\\
 & \quad+C\mathfrak{C}(\varsigma,\delta,\lambda)\cdot\mathbb{E}\bigg(1+\left\Vert u_{n,\epsilon}\right\Vert _{L_{m+1}(Q_{T})}^{m+1}+\left\Vert u_{n^{\prime},\epsilon}\right\Vert _{L_{m+1}(Q_{T})}^{m+1}\bigg)\\
 & \quad+C\varsigma^{\frac{2}{m+1}}\left(1+\mathbb{E}\left\Vert \nabla\left\llbracket \zeta_{n}\right\rrbracket (u_{n,\epsilon})\right\Vert _{L_{1}(Q_{T})}+\mathbb{E}\left\Vert \nabla\left\llbracket \zeta_{n^{\prime}}\right\rrbracket (u_{n^{\prime},\epsilon})\right\Vert _{L_{1}(Q_{T})}\right)\\
 & \quad+C\mathbb{E}\int_{t,x,y}\big|G_{\epsilon}(\tilde{u}(t,y),S(t))-G_{\epsilon}(u(t,x),S(t))\big|\varrho_{\varsigma}(x-y).
\end{align*}
Note that $G_{\epsilon}$ is Lipschitz continuous for fixed $\epsilon>0$.
Using Lemma \ref{lem:change for u} and Gr\"onwall's inequality, we
can eliminate the last term and obtain the $L_{1}$ estimate for $u_{n,\epsilon}$
and $u_{n^{\prime},\epsilon}$.

Without loss of generality, we can assume $n\leq n^{\prime}$. Taking
$\lambda=8/n$ and using Proposition \ref{prop:Assumptio for coef},
we have $R_{\lambda}>n$. We also choose $\vartheta>(m\land2)^{-1}$
and $\alpha\in(1/(2\vartheta),(m\land2)/2)$. Let $\delta=\varsigma^{2\vartheta}$.
With Theorem \ref{thm:esimate for u_n,epsilon}, we have
\[
\mathbb{E}\int_{t,x}\left|u_{n,\epsilon}(\tau,x)-u_{n^{\prime},\epsilon}(\tau,x)\right|\leq M_{1}(\varsigma)+M_{2}(\varsigma,n,n^{\prime})
\]
with
\[
M_{1}(\varsigma):=C\left(\sup_{|h|\leq\varsigma}\mathbb{E}\left\Vert \xi(\cdot)-\xi(\cdot+h)\right\Vert _{L_{1}(\mathbb{T}^{d})}+\varsigma^{4\vartheta\kappa}+\varsigma^{\bar{\kappa}}+\varsigma^{-2+4\alpha\vartheta}+\varsigma^{\frac{2}{m+1}}\right)
\]
and
\begin{align*}
M_{2}(\varsigma,n,n^{\prime}) & :=C\left(\mathbb{E}\left\Vert \xi-\xi_{n}\right\Vert _{L_{1}(\mathbb{T}^{d})}+\mathbb{E}\left\Vert \xi-\xi_{n^{\prime}}\right\Vert _{L_{1}(\mathbb{T}^{d})}\right)+C\varsigma^{-2}n^{-2}\\
 & \quad+C\varsigma^{-2}\Bigg(\mathbb{E}\left\Vert \mathbf{1}_{\{\left|u_{n,\epsilon}\right|\geq n\}}(1+\left|u_{n,\epsilon}\right|)\right\Vert _{L_{m}(Q_{T})}^{m}\\
 & \quad+\mathbb{E}\left\Vert \mathbf{1}_{\{\left|u_{n^{\prime},\epsilon}\right|\geq n\}}(1+\left|u_{n^{\prime},\epsilon}\right|)\right\Vert _{L_{m}(Q_{T})}^{m}\Bigg).
\end{align*}
Since $M_{1}(\varsigma)$ converges to 0 when $\varsigma\rightarrow0^{+}$,
for any $\varepsilon_{0}>0$, it is smaller than $\varepsilon_{0}$
for sufficiently small $\varsigma$. Then, for fixed $\varsigma$,
we choose $n_{0}$ big enough such that $M_{2}(\varsigma,n,n^{\prime})$
is smaller than $\varepsilon_{0}$ for all $n_{0}\leq n\leq n^{\prime}$.
Therefore, we have
\[
\mathbb{E}\left\Vert u_{n,\epsilon}-u_{n^{\prime},\epsilon}\right\Vert _{L_{1}(Q_{T})}\leq2\varepsilon_{0},\quad\forall n_{0}\leq n\leq n^{\prime},
\]
which indicates the sequence $\{u_{n,\epsilon}\}_{n\in\mathbb{N}}$
converges to a limit $u_{\epsilon}$ in $L_{1}(\Omega_{T};L_{1}(\mathbb{T}^{d}))$. By taking a subsequence, when $n\rightarrow\infty$, we may assume that $u_{n,\epsilon}\rightarrow u_{\epsilon}$ almost surely in $\Omega_{T}\times \mathbb{T}^d$. Moreover, Theorem \ref{thm:esimate for u_n,epsilon}
shows that the sequence $\{|u_{n,\epsilon}|^{q}\}_{n\in\mathbb{N}}$
is uniformly integrable on $\Omega_{T}\times\mathbb{T}^{d}$ for all
$q\in[0,m+1)$. 

If the right hand sides of (\ref{eq:u_e priori estimate1})-(\ref{eq:nu priori estimate})
are bounded, with the definition of $\xi_{n}$, the left hand sides
of the estimates in Theorems \ref{thm:esimate for u_n,epsilon} and
\ref{thm:estimate for L2 for G} are weak convergence in the corresponding
Banach space. By applying Banach-Saks Theorem and taking a subsequence,
the weak limits are the corresponding terms in the left hand side
of (\ref{eq:u_e priori estimate1})-(\ref{eq:nu priori estimate}).
Since
\[
\left\Vert f\right\Vert _{\mathfrak{B}}\leq\liminf_{n\rightarrow\infty}\left\Vert f_{n}\right\Vert _{\mathfrak{B}},\quad\forall f_{n}\in\mathfrak{B},\ f_{n}\rightharpoonup f,
\]
for all Banach space $\mathfrak{B}$, after taking inferior limit
to the estimates in Theorems \ref{thm:esimate for u_n,epsilon} and
\ref{thm:estimate for L2 for G}, we obtain estimates (\ref{eq:u_e priori estimate1})-(\ref{eq:nu priori estimate}).

Now we only need to verify that $u_{\epsilon}$ is an entropy solution
of $\Pi(\Phi,F-G_{\epsilon}(\cdot,S),\xi)$ in the sense of Definition
\ref{def:entropy solution without}. Firstly, Assertion (i) in Definition
\ref{def:entropy solution without} is a direct consequence of (\ref{eq:u_e priori estimate2}). 

As for Assertion (ii) in Definition \ref{def:entropy solution without},
for any $f\in C_{b}(\mathbb{R})$, using
\[
|[\zeta_{n}f](r)|\leq C\left\Vert f\right\Vert _{L_{\infty}}|r|^{(m+1)/2},\quad\forall r\in\mathbb{R},
\]
and $\partial_{x_{i}}\left\llbracket \zeta_{n}f\right\rrbracket (u_{n,\epsilon})=f(u_{n,\epsilon})\partial_{x_{i}}\left\llbracket \zeta_{n}\right\rrbracket (u_{n,\epsilon})$
and Theorem \ref{thm:esimate for u_n,epsilon}, we have
\[
\sup_{n}\mathbb{E}\int_{t}\left\Vert \left\llbracket \zeta_{n}f\right\rrbracket (u_{n,\epsilon})\right\Vert _{H^{1}(\mathbb{T}^{d})}^{2}<\infty.
\]
By taking a subsequence, we have that $\left\llbracket \zeta_{n}f\right\rrbracket (u_{n,\epsilon})$
converges weakly to some $v_{f}$ in $L_{2}(\Omega_{T};H^{1}(\mathbb{T}^{d}))$,
and $\left\llbracket \zeta_{n}\right\rrbracket (u_{n,\epsilon})$
converges weakly to some $v$ in $L_{2}(\Omega_{T};H^{1}(\mathbb{T}^{d}))$.
Then, with the pointwise convergence and uniform integrability of
$u_{n,\epsilon}$ and Proposition \ref{prop:Assumptio for coef},
we have $v_{f}=\left\llbracket \zeta f\right\rrbracket (u_{\epsilon})$
and $v=\left\llbracket \zeta\right\rrbracket (u_{\epsilon})$. Based
on the strong convergence of $f(u_{n,\epsilon})\phi$ and weak convergence
of $\left\llbracket \zeta_{n}\right\rrbracket (u_{n,\epsilon})$,
we have
\begin{align*}
\mathbb{E}\bigg[\mathbf{1}_{B}\int_{t,x}\partial_{x_{i}}\left\llbracket \zeta f\right\rrbracket (u_{\epsilon})\phi\bigg] & =\lim_{n\rightarrow\infty}\int_{t,x}\partial_{x_{i}}\left\llbracket \zeta_{n}f\right\rrbracket (u_{n,\epsilon})\phi\\
 & =\lim_{n\rightarrow\infty}\int_{t,x}f(u_{n,\epsilon})\partial_{x_{i}}\left\llbracket \zeta_{n}\right\rrbracket (u_{n,\epsilon})\phi\\
 & =\int_{t,x}f(u_{\epsilon})\partial_{x_{i}}\left\llbracket \zeta\right\rrbracket (u_{\epsilon})\phi,\quad\forall\phi\in C^{\infty}(\mathbb{T}^{d}),B\in\mathcal{F}.
\end{align*}

To prove Assertion (iii), denote $\eta$, $\varrho$ and $\phi$ as
test functions in Assertion (iii). Applying It\^{o}'s formula and using It\^{o}'s product rule, we have
\begin{align}
 & -\mathbb{E}\bigg[\mathbf{1}_{B}\int_{0}^{T}\int_{\mathbb{T}^{d}}\eta(u_{n,\epsilon})\partial_{t}\phi dxdt\bigg]\nonumber\\
 & =\mathbb{E}\Bigg\{\mathbf{1}_{B}\bigg[\int_{\mathbb{T}^{d}}\eta(\xi_{n})\phi(0)dx+\int_{0}^{T}\int_{\mathbb{T}^{d}}\llbracket \zeta_{n}^2 \eta^\prime \rrbracket (u_{n,\epsilon}) \Delta\phi dxdt\nonumber \\
 & \quad+\int_{0}^{T}\int_{\mathbb{T}^{d}}\eta^{\prime}(u_{n,\epsilon})\big(F(t,x,u_{n,\epsilon})-G_{\epsilon}(u_{n,\epsilon},S(t))\big)\phi dxdt\label{eq:Itofornepsilon} \\
 & \quad+\int_{0}^{T}\int_{\mathbb{T}^{d}}\left(\frac{1}{2}\eta^{\prime\prime}(u_{n,\epsilon})\sum_{k=1}^{\infty}|\sigma^{k}(u_{n,\epsilon})|^{2}\phi-\eta^{\prime\prime}(u_{n,\epsilon})|\nabla\left\llbracket \zeta_{n}\right\rrbracket (u_{n,\epsilon})|^{2}\phi\right)dxdt\nonumber \\
 & \quad+\sum_{k=1}^{\infty}\int_{0}^{T}\int_{\mathbb{T}^{d}}\eta^{\prime}(u_{n,\epsilon})\phi\sigma^{k}(u_{n,\epsilon})dxdW_{t}^{k}\bigg]\Bigg\},\nonumber 
\end{align}
Similar to the proof of Assertion (ii), we have that $\partial_{x_{i}}\left\llbracket \zeta_{n}\sqrt{\eta^{\prime\prime}}\right\rrbracket (u_{n})$
 converges weakly to $\partial_{x_{i}}\left\llbracket \zeta\sqrt{\eta^{\prime\prime}}\right\rrbracket (u)$
in $L_{2}(\Omega_{T};L_{2}(\mathbb{T}^{d}))$, which implies
\begin{align*}
 & \mathbb{E}\Bigg[\mathbf{1}_{B}\int_{0}^{T}\int_{\mathbb{T}^{d}}\eta^{\prime\prime}(u_{\epsilon})|\nabla\left\llbracket \zeta\right\rrbracket (u_{\epsilon})|^{2}\phi dxdt\Bigg]\\
 & \leq\liminf_{n\rightarrow\infty}\mathbb{E}\Bigg[\mathbf{1}_{B}\int_{0}^{T}\int_{\mathbb{T}^{d}}\eta^{\prime\prime}(u_{n,\epsilon})|\nabla\left\llbracket \zeta_{n}\right\rrbracket (u_{n,\epsilon})|^{2}\phi dxdt\Bigg].
\end{align*}
Therefore, taking inferior limit on (\ref{eq:Itofornepsilon}) and
using Assumptions \ref{assu:assumption for phi}-\ref{assu:inhomogeneous function}
and the convergence of $u_{n,\epsilon}$, we acquire the entropy formulation
in (iii). The proof is complete.
\end{proof}
\begin{rem}
\label{rem:Star_for_u_e}If furthermore $\mathbb{E}\left\Vert \xi\right\Vert _{L_{2}(\mathbb{T}^{d})}^{4}<\infty$,
applying Lemma \ref{lem:appro for star} and Theorem \ref{thm:star property},
we have the $(\star)$-property of $u_{\epsilon}$, and the constant
$C$ in Definition \ref{def:star property} is independent of $\epsilon$.
\end{rem}

\begin{lem}
\label{lem:Lemma for comparison}Let Assumptions \ref{assu:assumption for phi}-\ref{assu:Assuption for barrier}
hold. For each $\epsilon_{1}>\epsilon_{2}>0$, let $u_{\epsilon_{1}}$
and $u_{\epsilon_{2}}$ be the entropy solutions constructed in Theorem
\ref{The:ex and uni for u_e}. Then, we have
\[
u_{\epsilon_{1}}\geq u_{\epsilon_{2}}\geq0,\quad\textrm{a.s.}\ (\omega,t,x)\in\Omega_{T}\times\mathbb{T}^{d}.
\]
\end{lem}

\begin{proof}
For each $n\in\mathbb{N}$, denote $u_{n,\epsilon_{1}}$ and $u_{n,\epsilon_{2}}$
are the $L_{2}$-solutions of $\Pi(\Phi_{n},F-G_{\epsilon_{1}}(\cdot,S),\xi_{n})$
and $\Pi(\Phi_{n},F-G_{\epsilon_{2}}(\cdot,S),\xi_{n})$, respectively.
Based on the proof of Theorem \ref{The:ex and uni for u_e}, by repectedly taking subsequences, when $n\rightarrow\infty$, we can
assume $u_{n,\epsilon_{1}}\rightarrow u_{\epsilon_1}$ and $u_{n,\epsilon_{2}}\rightarrow u_{\epsilon_2}$ almost surely in $\Omega_{T}\times\mathbb{T}^d$. Therefore, we only need to prove
\[
u_{n,\epsilon_{1}}\geq u_{n,\epsilon_{2}}\geq0,\quad\textrm{a.s.}\ (\omega,t,x)\in\Omega_{T}\times\mathbb{T}^{d},
\]
while the second inequality is shown in Lemma \ref{Lem:nonnegative of u_n,e}.

For the first inequality, Using Remark \ref{rem:L2equalentropy},
we have that $L_{2}$-solutions $u_{n,\epsilon_{1}}$ and $u_{n,\epsilon_{2}}$
are also entropy solutions of the corresponding equations. Therefore,
we apply Theorem \ref{lem:L1} with $u=u_{n,\epsilon_{2}}$ and $\tilde{u}=u_{n,\epsilon_{1}}$
and take $\Phi=\tilde{\Phi}:=\Phi_{n}$, $\xi=\tilde{\xi}:=\xi_{n}$,
$G(t,x,r):=G_{\epsilon_{2}}(r,S(t))$ and $\tilde{G}(t,x,r):=G_{\epsilon_{1}}(r,S(t))$.
Then, we have for all $\tau\in[0,T]$, $\varsigma,\delta\in(0,1)$,
$\lambda\in[0,1]$ and $\alpha\in(0,1\land(m/2))$,
\begin{align}
 & \mathbb{E}\int_{x,y}\left(u_{n,\epsilon_{2}}(\tau,x)-u_{n,\epsilon_{1}}(\tau,y)\right)^{+}\varrho_{\varsigma}(x-y)\nonumber\\
 & \leq\mathbb{E}\int_{x,y}\left(\xi_{n}(x)-\xi_{n}(y)\right)^{+}\varrho_{\varsigma}(x-y)\nonumber \\
 & \quad+C\mathfrak{C}(\varsigma,\delta)\mathbb{E}\left(1+\left\Vert u_{n,\epsilon_{2}}\right\Vert _{L_{m+1}(Q_{T})}^{m+1}+\left\Vert u_{n,\epsilon_{1}}\right\Vert _{L_{m+1}(Q_{T})}^{m+1}\right)\label{eq:compareL1} \\
 & \quad+C\mathbb{E}\int_{0}^{\tau}\int_{x,y}\left(u_{n,\epsilon_{2}}(t,x)-u_{n,\epsilon_{1}}(t,y)\right)^{+}\varrho_{\varsigma}(x-y)dt\nonumber \\
 & \quad+C\mathbb{E}\int_{0}^{\tau}\int_{x,y}\mathbf{1}_{\{u_{n,\epsilon_{2}}(t,x)\geq u_{n,\epsilon_{1}}(t,y)\}}\nonumber \\
 & \quad\cdot\left(\frac{1}{\epsilon_{1}}\left(u_{n,\epsilon_{1}}(t,y)-S(t)\right)^{+}-\frac{1}{\epsilon_{2}}\left(u_{n,\epsilon_{2}}(t,x)-S(t)\right)^{+}\right)^{+}\varrho_{\varsigma}(x-y)dt.\nonumber 
\end{align}
Since
\begin{align*}
 & \frac{1}{\epsilon_{1}}(u_{n,\epsilon_{1}}(t,y)-S(t))^{+}-\frac{1}{\epsilon_{2}}(u_{n,\epsilon_{2}}(t,x)-S(t))^{+}\\
 & \leq\frac{1}{\epsilon_{2}}\left(u_{n,\epsilon_{1}}(t,y)-u_{n,\epsilon_{2}}(t,x)\right)^{+},
\end{align*}
using the non-negativity of $\varrho_{\varsigma}$, the last term
on the right hand side of (\ref{eq:compareL1}) is no more than $0$.
On the other hand, choose $\vartheta>(m\wedge2)^{-1}$ and then $\alpha<1\wedge(m/2)$
such that $-2+(2\alpha)(2\vartheta)>0$. Let $\delta=\varsigma^{2\vartheta}$
then yields $\mathfrak{C}(\varsigma,\delta)\rightarrow0$ as $\varsigma\rightarrow0^{+}$.
Therefore, with the continuity of translations in $L_{1}$ and taking
the limit $\varsigma\rightarrow0^{+}$, we have
\begin{align*}
 & \mathbb{E}\int_{x,y}(u_{n,\epsilon_{2}}(\tau,x)-u_{n,\epsilon_{1}}(\tau,x))^{+}\\
 & \leq C\mathbb{E}\int_{0}^{\tau}\int_{x,y}(u_{n,\epsilon_{2}}(t,x)-u_{n,\epsilon_{1}}(t,x))^{+}dt.
\end{align*}
Using Gr\"onwall's inequality, we have the desired result.
\end{proof}
\begin{thm}
\label{thm:existfor u_v}Let Assumptions \ref{assu:assumption for phi}-\ref{assu:Assuption for barrier}
hold. Then, the obstacle problem $\Pi_{S}(\Phi,F,\xi)$ has an entropy
solution $(u,\nu)$ in the sense of Definition \ref{def:entropy solution with ob}.
\end{thm}

\begin{proof}
Let $\{\epsilon_{i}\}_{i\in\mathbb{N}}$ be a monotone decreasing
sequence such that $\lim_{i\rightarrow\infty}\epsilon_{i}=0$. From
Theorem \ref{The:ex and uni for u_e} and Lemma \ref{lem:Lemma for comparison},
equation $\Pi(\Phi,F-G_{\epsilon_{i}}(\cdot,S),\xi)$ has an entropy
solution $u_{\epsilon_{i}}$, and the functions $u_{\epsilon_{i}}$ almost surely decrease
to a limit $u$ as $i\rightarrow\infty$. This is also a strong convergence
in $L_{m+1}(\Omega;C([0,T];L_{m+1}(\mathbb{T}^{d})))$ based on the
dominated convergence theorem and (\ref{eq:u_e priori estimate2}),
and we have
\[
\mathbb{E}\sup_{t\leq T}\left\Vert u(t)\right\Vert _{L_{m+1}(\mathbb{T}^{d})}^{m+1}\leq C\left(1+\mathbb{E}\left\Vert \xi\right\Vert _{L_{m+1}(\mathbb{T}^{d})}^{m+1}\right).
\]

On the other hand, using estimate (\ref{eq:nu priori estimate}) and
taking a subsequence, we have that the sequence $\{G_{\epsilon_{i}}(u_{\epsilon_{i}},S)\}_{i\in\mathbb{N}}$
converges weakly to some function $\nu\in L_{2}(\Omega_{T}\times\mathbb{T}^{d})$
as $i\rightarrow\infty$. The non-negativity of $\nu$ is easily obtained
by taking test function in $L_{2}(\Omega_{T}\times\mathbb{T}^{d})$.
Applying Banach-Saks Theorem and taking the subsequence again, we
have
\[
\mathbb{E}\left\Vert \nu\right\Vert _{L_{2}(Q_{T})}^{2}\leq C\left(1+\mathbb{E}\left\Vert \xi\right\Vert _{L_{2}(\mathbb{T}^{d})}^{2}\right).
\]
Now we verify that $(u,\nu)$ is an entropy solution of obstacle problem
$\Pi_{S}(\Phi,F,\xi)$. Note that Assertion (i) of Definition \ref{def:entropy solution with ob}
has been proved, and Assertion (ii) and (iii) can be verified as in
the proof of Theorem \ref{The:ex and uni for u_e} via the strong
convergences of $u_{\epsilon_{i}}$ and $\eta^{\prime}(u_{\epsilon_{i}})$
and the weak convergence of $G_{\epsilon_{i}}(u_{\epsilon_{i}},S)$. 

For Assertion (iv), using estimate (\ref{eq:nu priori estimate})
and the strong convergence of $u_{\epsilon_{i}}$, we have
\begin{align*}
\mathbb{E}\left\Vert (u-S)^{+}\right\Vert _{L_{2}(Q_{T})}^{2} & =\lim_{i\rightarrow\infty}\mathbb{E}\left\Vert (u_{\epsilon_{i}}-S)^{+}\right\Vert _{L_{2}(Q_{T})}^{2}\\
 & \leq\lim_{i\rightarrow\infty}C\epsilon_{i}\left(1+\mathbb{E}\left\Vert \xi\right\Vert _{L_{2}(\mathbb{T}^{d})}^{2}\right)=0.
\end{align*}
Therefore, we have $u\leq S$ almost everywhere in $Q_{T}$, almost
surely.

Furthermore, using the strong convergence of $u_{\epsilon_{i}}-S$
and the weak convergence of $(u_{\epsilon_{i}}-S)^{+}/{\epsilon_{i}}$,
we obtain
\[
\mathbb{E}\int_{Q_{T}}(u-S)\nu dtdx=\lim_{i\rightarrow\infty}\mathbb{E}\int_{Q_{T}}(u_{\epsilon_{i}}-S)\frac{1}{\epsilon_{i}}(u_{\epsilon_{i}}-S)^{+}dtdx\geq0.
\]
Since $\nu\geq0$ and $u\leq S$, we have
\[
\mathbb{E}\int_{Q_{T}}(u-S)\nu dtdx\leq0.
\]
Combining these two inequalities, we have that the entropy solution
$(u,\nu)$ satisfies the Skohorod condition.
\end{proof}
\begin{rem}
\label{rem:starprofor u}If furthermore $\mathbb{E}\left\Vert \xi\right\Vert _{L_{2}(\mathbb{T}^{d})}^{4}<\infty$,
then Lemma \ref{lem:appro for star} and Remark \ref{rem:Star_for_u_e}
show that $u$ in Theorem \ref{thm:existfor u_v} has $(\star)$-property,
which will be used in the proof of uniqueness.
\end{rem}

\section{Uniqueness of solution}
\begin{thm}
\label{thm:uniqueentropy}Let Assumptions \ref{assu:assumption for phi}-\ref{assu:Assuption for barrier}
hold. Suppose that $(u,\nu)$ and $(\tilde{u},\tilde{\nu})$ are two
entropy solutions of the obstacle problem $\Pi_{S}(\Phi,F,\xi)$ and
$\Pi_{S}(\Phi,F,\tilde{\xi})$, respectively. Moreover, $u$ has $(\star)$-property.
Then, we have
\[
\underset{t\in[0,T]}{\mathrm{ess\ sup}}\ \mathbb{E}\int_{x}|u(t,x)-\tilde{u}(t,x)|\leq C\mathbb{E}\int_{x}|\xi(x)-\tilde{\xi}(x)|
\]
for a constant $C$ depending only on $K$, $N_{0}$, $d$ and $T$.
If furthermore $\xi=\tilde{\xi}$, we have $u=\tilde{u}$ almost everywhere
in $Q_{T}$, almost surely.
\end{thm}

\begin{proof}
To get rid of using the $(\star)$-property of $\tilde{u}$, we need
to adjust the proof of Lemma \ref{lem:Lemma for L1} as in the proof
of \cite[Theorem 4.1]{dareiotis2019entropy}. We take $\eta_{\delta}\in C^{2}(\mathbb{R})$
such that
\[
\eta_{\delta}(0)=\eta_{\delta}^{\prime}(0)=0,\ \ \ \eta_{\delta}^{\prime\prime}(r)=\rho_{\delta}(|r|).
\]
Therefore, we have
\[
\left|\eta_{\delta}(r)-|r|\right|\leq\delta,\ \ \ \mathrm{supp}\ \eta_{\delta}^{\prime\prime}\subset[\delta,\delta],\ \ \ \int_{\mathbb{R}}\eta_{\delta}^{\prime\prime}(r)dr\leq2,\ \ \ |\eta_{\delta}^{\prime\prime}|\leq2\delta^{-1}.
\]

Based on the symmetry of $\eta_{\delta}$, we apply entropy formulation
(\ref{eq:entropy formula-1}) on $\eta_{\delta}(r-a)$ with $(r,a):=\big(u(t,x),\tilde{u}(s,y)\big)$
or $(r,a):=\big(\tilde{u}(s,y),u(t,x)\big)$ instead of both $\eta_{\delta}(r-a)$
and $\eta_{\delta}(a-r)$ in Lemma \ref{lem:Lemma for L1}. By applying
the $(\star)$-property of $u$ and taking the limit $\theta\rightarrow0^{+}$,
for all $\varsigma,\delta\in(0,1]$, we have
\begin{align}
 & -\mathbb{E}\int_{t,x,y}\eta_{\delta}(u-\tilde{u})\partial_{t}\phi_{\varsigma}\nonumber\\
 & \leq\mathbb{E}\int_{t,x,y}\llbracket \zeta^2 \eta_{\delta}^{\prime}(\cdot-\tilde{u})\rrbracket(u) \Delta_{x}\phi_{\varsigma}+\mathbb{E}\int_{t,x,y}\llbracket {\zeta}^2 \eta_{\delta}^{\prime}(u-\cdot)\rrbracket (\tilde{u})\Delta_{y}\phi_{\varsigma}\nonumber \\
 & \quad+\mathbb{E}\int_{t,x,y}\eta_{\delta}^{\prime}(u-\tilde{u})\Big[\big(F(t,x,u)-\nu\big)-\big(F(t,y,\tilde{u})-\tilde{\nu}\big)\Big]\phi_{\varsigma}\label{eq:diffL1 u tildeu} \\
 & \quad+\mathbb{E}\int_{t,x,y}\left(\frac{1}{2}\eta_{\delta}^{\prime\prime}(u-\tilde{u})\sum_{k=1}^{\infty}|\sigma^{k}(u)|^{2}\phi_{\varsigma}-\eta_{\delta}^{\prime\prime}(u-\tilde{u})|\nabla_{x}\left\llbracket \zeta\right\rrbracket (u)|^{2}\phi_{\varsigma}\right)\nonumber \\
 & \quad+\mathbb{E}\int_{t,x,y}\left(\frac{1}{2}\eta_{\delta}^{\prime\prime}(u-\tilde{u})\sum_{k=1}^{\infty}|\sigma^{k}(\tilde{u})|^{2}\phi_{\varsigma}-\eta_{\delta}^{\prime\prime}(u-\tilde{u})|\nabla_{y}\left\llbracket \zeta\right\rrbracket (\tilde{u})|^{2}\phi_{\varsigma}\right)\nonumber \\
 & \quad-\sum_{k=1}^{\infty}\mathbb{E}\int_{t,x,y}\phi_{\varsigma}\sigma^{k}(u)\sigma^{k}(\tilde{u})\eta_{\delta}^{\prime\prime}(u-\tilde{u}).\nonumber 
\end{align}
Here $u=u(t,x)$, $\nu=\nu(t,x)$, $\tilde{u}=\tilde{u}(t,y)$ and
$\tilde{\nu}=\tilde{\nu}(t,y)$, and the definition of $\phi_{\varsigma}$ can be found in Lemma \ref{lem:Lemma for L1}. Since $\eta_{\delta}^{\prime}$ is odd and monotone, using the Skohorod condition for $(u,\nu)$ and $(\tilde{u},\tilde{\nu})$,
we have
\begin{align*}
\mathbb{E}\int_{t,x,y}\eta_{\delta}^{\prime}(u-\tilde{u})(\tilde{\nu}-\nu)\phi_{\varsigma} & =\mathbb{E}\int_{t,x,y}\eta_{\delta}^{\prime}(S-\tilde{u})\mathbf{1}_{\{u=S,\tilde{u}\neq S\}}(-\nu)\phi_{\varsigma}\\
 & \quad+\mathbb{E}\int_{t,x,y}\eta_{\delta}^{\prime}(u-S)\mathbf{1}_{\{u\neq S,\tilde{u}=S\}}\tilde{\nu}\phi_{\varsigma}\leq0.
\end{align*}
The estimates for other terms in the right hand side of (\ref{eq:diffL1 u tildeu})
are similar to the proof of \cite[Theorem 4.1]{dareiotis2019entropy}
or Lemma \ref{lem:Lemma for L1}. Since
\[
\left|\mathbb{E}\int_{t,x,y}\eta_{\delta}(u-\tilde{u})\partial_{t}\phi_{\varsigma}-\mathbb{E}\int_{t,x,y}|u-\tilde{u}|\partial_{t}\phi_{\varsigma}\right|\leq C\delta
\]
 $\kappa\in(0,1/2]$, we have
\begin{align}
 & -\mathbb{E}\int_{t,x,y}|u-\tilde{u}|\partial_{t}\phi_{\varsigma}\label{eq:diffu tildeu}\\
 & \leq C\mathbb{E}\int_{t,x,y}|u-\tilde{u}|\phi_{\varsigma}+C\mathfrak{C}(\varsigma,\delta)\mathbb{E}\left(1+\left\Vert u\right\Vert _{L_{m+1}(Q_{T})}^{m+1}+\left\Vert \tilde{u}\right\Vert _{L_{m+1}(Q_{T})}^{m+1}\right).\nonumber 
\end{align}
For $\vartheta>(m\land2)^{-1}$ and $\alpha\in(1/(2\vartheta),1\land(m/2))$,
we have $-1+2\alpha\vartheta>0$. Choose $\delta=\varsigma^{2\vartheta}$
such that $\mathfrak{C}(\varsigma,\delta)\rightarrow0$ as $\varsigma\rightarrow0^{+}$.
With the continuity of translations, we have
\[
-\mathbb{E}\int_{t,x}|u(t,x)-\tilde{u}(t,x)|\partial_{t}\varphi(t)\leq C\mathbb{E}\int_{t,x}|u(t,x)-\tilde{u}(t,x)|\varphi(t)
\]
for a constant $C$ depending only on $K$, $N_{0}$, $d$ and $T$.
Similar to the proof of (\ref{eq:ineq_givenphis}), we have
\begin{align*}
 & \mathbb{E}\int_{\mathbb{T}^{d}}\left|u(\tau,x)-\tilde{u}(\tau,x)\right|dx\\
 & \leq\mathbb{E}\int_{\mathbb{T}^{d}}\big|\xi(x)-\tilde{\xi}(x)\big|dx+C\mathbb{E}\int_{0}^{\tau}\int_{\mathbb{T}^{d}}\left|u(t,x)-\tilde{u}(t,x)\right|dxdt,\quad\mathrm{a.e.}\ \tau\in[0,T].
\end{align*}
Using Gr\"onwall's inequality, we prove the theorem.
\end{proof}
Now we prove our main theorem.
\begin{proof}[Proof of Theorem \ref{thm:maintheorem}]
The existence is referred to Theorem \ref{thm:existfor u_v}. For
the uniqueness, we define $\xi_{n}$ as in \eqref{defn for xi_n}.
Denote $(u_{n},\nu_{n})$ as the entropy solution of the obstacle
problem $\Pi_{S}(\Phi,F,\xi_{n})$ constructed in Theorem \ref{thm:existfor u_v}.
From Remark \ref{rem:starprofor u}, the function $u_{n}$ has $(\star)$-property.
Then, Theorem \ref{thm:uniqueentropy} indicates the uniqueness of
$u_{n}$.

On the other hand, for any entropy solution $(u,\nu)$ to the obstacle
problem $\Pi_{S}(\Phi,F,\xi)$, applying Theorem \ref{thm:uniqueentropy}
again, we have
\[
\underset{t\in[0,T]}{\mathrm{ess\ sup}}\ \mathbb{E}\int_{\mathbb{T}^{d}}|u_{n}(t,x)-u(t,x)|dx\leq C\mathbb{E}\int_{\mathbb{T}^{d}}|\xi_{n}(x)-\xi(x)|dx
\]
for a constant $C$ depending only on $K$, $N_{0}$, $d$ and $T$.
Therefore, $u_{n}$ converges to $u$ in $L_{1}(\Omega_{T}\times\mathbb{T}^{d})$
when $n\rightarrow\infty$. Then, the uniqueness of $u_{n}$ gives
the uniqueness of $u$.

Now, we apply entropy formulation (\ref{eq:entropy formula-1}) with
the functions $\eta(r):=r$
and $\eta(r):=-r$. By taking
expectations and combining these two inequalities, we have
\begin{align*}
-\mathbb{E}\int_{0}^{T}\int_{\mathbb{T}^{d}}u\partial_{t}\phi dxdt & =\mathbb{E}\Bigg[\int_{\mathbb{T}^{d}}\xi\phi(0)dx+\int_{0}^{T}\int_{\mathbb{T}^{d}}\Phi(u)\Delta\phi dxdt\\
 & \quad+\int_{0}^{T}\int_{\mathbb{T}^{d}}\big(F(t,x,u)-\nu\big)\phi dxdt\Bigg].
\end{align*}
With the uniqueness of $u$, if there exists another entropy solution
$(u,\tilde{\nu})$ to the obstacle problem $\Pi_{S}(\Phi,F,\xi)$,
we have
\[
\int_{0}^{T}\int_{\mathbb{T}^{d}}\nu\phi dxdt=\int_{0}^{T}\int_{\mathbb{T}^{d}}\nu\phi dxdt,
\]
for all test function $\phi:=\varphi\varrho\geq0$, where $(\varphi,\varrho)\in C_{c}^{\infty}([0,T))\times C^{\infty}(\mathbb{T}^{d})$.
Therefore, we have $\nu=\tilde{\nu}$ almost everywhere in $Q_{T}$,
almost surely.
\end{proof}
\bibliographystyle{plain}
\bibliography{bi}

\end{document}